\documentclass[reqno]{amsart}
\usepackage[foot]{amsaddr}

\usepackage{graphicx} %
\usepackage{amsmath,amssymb,amsthm,amscd,mathtools}

\usepackage{setspace,float}
\usepackage{multicol}
\usepackage{fullpage}
\usepackage{tcolorbox,xcolor}
\usepackage{algorithm,algpseudocode}

\algnewcommand\algorithmicswitch{\textbf{switch}}
\algnewcommand\algorithmiccase{\textbf{case}}
\algnewcommand\algorithmicassert{\texttt{assert}}
\algnewcommand\Assert[1]{\State \algorithmicassert(#1)}%
\algdef{SE}[SWITCH]{Switch}{EndSwitch}[1]{\algorithmicswitch\ #1\ \algorithmicdo}{\algorithmicend\ \algorithmicswitch}%
\algdef{SE}[CASE]{Case}{EndCase}[1]{\algorithmiccase\ #1}{\algorithmicend\ \algorithmiccase}%
\algtext*{EndSwitch}%
\algtext*{EndCase}%

\usepackage{bm,blkarray}

\usepackage{mleftright,enumitem}
\usepackage{hhline}
\usepackage{caption,subcaption}
\usepackage{tikz,tikz-cd}

\usepackage[hidelinks]{hyperref}
\usepackage[nameinlink]{cleveref}

\usepackage{refcount}

\newcommand{\R}{{\mathbb{R}}}
\newcommand{\Z}{{\mathbb{Z}}}
\newcommand{\N}{{\mathbb{N}}}

\newcommand{\sph}{{\mathcal{S}}}

\newcommand{\eps}{\varepsilon}

\newcommand{\wt}{\widetilde}
\newcommand{\wh}{\widehat}

\newcommand{\abs}[1]{\mleft|#1\mright|}
\newcommand{\sabs}[1]{|#1|}
\newcommand{\magn}[1]{\left\|#1\right\|}
\newcommand{\smagn}[1]{\|#1\|}
\newcommand{\mmagn}[1]{\big\|#1\big\|}
\newcommand{\pare}[1]{\mleft(#1\mright)}

\newcommand{\set}[1]{{\left\{{#1}\right\}}}

\newcommand{\bmat}[1]{\begin{bmatrix}#1\end{bmatrix}}
\newcommand{\pmat}[1]{\begin{pmatrix}#1\end{pmatrix}}

\newcommand{\ceil}[1]{\mleft\lceil#1\mright\rceil}

\newcommand{\floor}[1]{\mleft\lfloor#1\mright\rfloor}

\newcommand{\spliteq}[2]{\begin{equation}#1\begin{split}#2\end{split}\end{equation}}
\newcommand{\eq}[1]{\begin{equation}{#1}\end{equation}}

\DeclareMathOperator*{\E}{\mathbb{E}}

\DeclareMathOperator{\col}{Col}

\DeclareMathOperator{\spn}{span}

\DeclareMathOperator{\tr}{tr}

\DeclareMathOperator{\diag}{diag}

\DeclareMathOperator{\Haar}{Haar}

\newcommand{\at}[1]{^{(#1)}}

\newcommand{\toth}{^{\textnormal{th}}}

\newtheorem{theorem}{Theorem}[section]
\newtheorem{lemma}[theorem]{Lemma}
\newtheorem{proposition}[theorem]{Proposition}

\AddToHook{env/lemma/begin}{\crefalias{theorem}{lemma}}
\AddToHook{env/corollary/begin}{\crefalias{theorem}{corollary}}
\AddToHook{env/proposition/begin}{\crefalias{theorem}{proposition}}
\AddToHook{env/definition/begin}{\crefalias{theorem}{definition}}
\AddToHook{env/example/begin}{\crefalias{theorem}{example}}

\crefname{theorem}{Theorem}{Theorems}
\crefname{corollary}{Corollary}{Corollaries}
\crefname{lemma}{Lemma}{Lemmas}
\crefname{equation}{}{}
\crefname{section}{Section}{Sections}
\crefname{subsection}{Subsection}{Subsections}
\crefname{figure}{Figure}{Figures}
\crefname{proposition}{Proposition}{Propositions}

\makeatletter
\newtheorem*{rep@theorem}{\rep@title}
\newcommand{\newreptheorem}[2]{%
\newenvironment{rep#1}[1]{%
 \def\rep@title{#2 \ref{##1}}%
 \begin{rep@theorem}}%
 {\end{rep@theorem}}}
\makeatother

\newreptheorem{theorem}{Theorem}
\newreptheorem{coro}{Corollary}

\theoremstyle{definition}
\newtheorem{definition}[theorem]{Definition}

\usepackage{amssymb}
\usepackage{amscd}
\usepackage{amsthm}
\usepackage{setspace}
\usepackage{graphicx}
\usepackage{tcolorbox} 
\usepackage{subcaption} 

\usepackage{tikz-cd}
\usepackage{bm,blkarray}
\numberwithin{equation}{section}
\usepackage{algorithm}
\usepackage{algpseudocode}

\usepackage{tabularray}
\UseTblrLibrary{amsmath} 

\usepackage[nameinlink]{cleveref}

\usepackage{mathtools}
\usepackage[tableposition=top]{caption}
\usepackage{booktabs,dcolumn}

\usepackage{refcount}

\newcommand\Y{\mathbf{Y}}
\newcommand\W{\mathbf{W}}

\newcommand\Q{\mathbb{Q}}

\newcommand\Col{{\operatorname{Col}}}

\renewcommand\epsilon\varepsilon

\newcommand{\norm}[1]{{\left\lVert #1 \right\rVert}}

\newcommand{\GL}{\mathop{\rm GL}}

\newcommand{\up}[1]{^{\smash{#1}}}
\newcommand{\RandPP}{\textsc{RandPP}}
\newcommand{\VolPP}{\textsc{Vol}}

\usepackage{xcolor}      %
\usepackage{listings}    %

\definecolor{codegreen}{rgb}{0,0.6,0}
\definecolor{codegray}{rgb}{0.5,0.5,0.5}
\definecolor{codepurple}{rgb}{0.58,0,0.82}
\definecolor{white}{rgb}{1,1,1}

\lstdefinestyle{mystyle}{
    backgroundcolor=\color{white}, 
    commentstyle=\color{codegreen},
    keywordstyle=\color{magenta}, numberstyle=\tiny\color{codegray}, stringstyle=\color{codepurple}, basicstyle=\ttfamily\footnotesize, breakatwhitespace=false,  breaklines=true,                 
    captionpos=b,                    
    keepspaces=true,                 
    numbers=left,                    
    numbersep=5pt,                  
    showspaces=false,                
    showstringspaces=false,
    showtabs=false,                  
    tabsize=2,
    frame=single,
    rulecolor=\color{black},
    title=\lstname
}

\title[Entry growth in Gaussian elimination]{Entry growth in Gaussian elimination}

\author{Rikhav Shah}
\address{Department of Mathematics, Massachusetts Institute of Technology, Cambridge, MA 02139 USA.}
\email{rdshah@mit.edu}
\author{John Urschel}
\email{urschel@mit.edu}

\newtheorem{example}[theorem]{Example}

\newtheorem{remark}[theorem]{Remark}

\newcommand{\maxnorm}[1]{\magn{#1}_{\max}}

\newcommand{\mmaxnorm}[1]{\mmagn{#1}_{\max}}
\DeclareMathOperator{\Mat}{Mat}
\DeclareMathOperator{\orthog}{O}

\DeclareMathOperator{\growth}{growth}
\newcommand{\one}{\mathbf1}
\newcommand{\zero}{\mathbf0}
\DeclareMathOperator{\PP}{PP}
\DeclareMathOperator{\CP}{CP}
\DeclareMathOperator{\RP}{RP}
\DeclareMathOperator{\nnz}{nnz}

\DeclareMathOperator{\vvec}{vec}

\begin{document}

\begin{abstract}
Gaussian elimination is one of the oldest algorithms in mathematics, and the most popular method for solving an unstructured linear system. Its stability in finite precision is controlled by its growth factor, which measures how large the entries produced during elimination can become. Understanding the worst-case behavior of this quantity has been a central problem in numerical analysis since the 1940s.

Here we make a significant leap in that understanding, settling several open problems. In particular, we determine the asymptotic behavior of the maximum growth factor under complete and rook pivoting, proving that both are quasi-polynomial in dimension. We also show that the exponential growth under partial pivoting persists for sparse matrices and that randomized partial pivoting suffers the same instability. In contrast, we show that every non-singular matrix has a row permutation with polynomial growth, though finding the optimal row permutation is NP-hard.
\end{abstract}

\maketitle

\section{Introduction}

Solving a linear system $A\bm{x}=\bm{b}$ is one of the oldest problems in mathematics, and the technique now known as Gaussian elimination is one of the oldest and most popular approaches. In Gaussian elimination, an $n \times n$ matrix $A$ with non-zero leading principal minors is factored into the product of a lower triangular matrix $L$ and an upper triangular matrix $U$ by a sequence of rank-one transformations, resulting in the sequence of matrices $A\at0 = A$ and $A\at k = \big(a_{ij}\at k\big)_{i,j=1}^{n-k}$ for $k$ equals $1$ to $n-1$, satisfying
\begin{equation}\label{a0}
A\at k = Z_k - Y_k W_k^{-1} X_k,  \qquad \text{where } A = \,\begin{blockarray}{ccc}
         {\scriptstyle k} & {\scriptstyle n-k} &  \\
       \begin{block}{[cc]c}
         W_k & X_k &{ \scriptstyle k} \\ 
          Y_k & Z_k & {\scriptstyle n-k} \\
       \end{block}
     \end{blockarray}. 
     \end{equation}
The sequence $A\at0,\ldots,A\at n$ is shift-invariant, i.e., $A^{(k+\ell)}$ is the $\ell\toth$ step of Gaussian elimination applied to $A\at k$, and defining $A^{(k+1)}$ as a function of $A\at k$ recovers the Gaussian elimination algorithm itself.
The resulting LU factorization of $A$ is encoded by the first row and column of each of the iterates. In particular, 
\[L_{ij}  = \frac{a\at{j-1}_{i-j+1,1}}{a\at{j-1}_{11}} \quad \text{for} \quad i \ge j \qquad \text{and} \qquad U_{ij} = a\at{i-1}_{1,j-i+1} \quad \text{for} \quad i \le j.\]
The LU factorization of a matrix, under the normalization that $L$ is unit lower triangular, is unique if it exists. Unfortunately, a matrix does not have an LU factorization if one of its leading submatrices $W_k$ is singular, in which case a permutation of the rows (and/or columns) may be necessary. What's more, performing these computations in finite precision can produce cascading round-off errors. This error is controlled by a quantity known as the growth factor
\[\growth(A)=\max\set{\maxnorm L, \frac{\maxnorm U}{\maxnorm A}},\]
where $\maxnorm\cdot$ is the entrywise maximum norm (if no LU factorization exists, set $\growth(A)=+\infty$). 
In order for the Gaussian elimination procedure to compute the LU factorization of a matrix in a numerically stable way, the growth factor of that matrix must not be too large. Roughly speaking, the number of bits of precision required for stable computation scales as $\log(\growth(A))$, see \cite[Theorem 3.3.1]{b0} and \cite[Theorem 9.3 \& Lemma 9.6]{b1} for details. 

For this reason, in practice the matrix $A$ is replaced by the matrix $P A Q$, where $P$ and $Q$ are permutation matrices that can be chosen quickly, with the idea that $\growth(P A Q)$ may be significantly smaller than $\growth(A)$. The application of $P$ and $Q$ can then be undone by downstream tasks after Gaussian elimination is completed. The rule governing which permutation matrices $P$ and $Q$ are chosen for a given matrix $A$ is referred to as a pivoting strategy, and the top-left entry of each step of Gaussian elimination is referred to as the pivot. In this work, we prove several results related to the growth factor of Gaussian elimination under various pivoting strategies.

\subsection{Pivoting Strategies} Before presenting our contributions, we formally define the pivoting strategies under consideration. Complete pivoting, which takes the largest magnitude entry of the matrix as the pivot, is the most traditional pivoting strategy. Its analysis dates back to the seminal 1947 {\it Numerical Inverting of Matrices of High Order} by Goldstine and von Neumann \cite{b2}, often referred to as the ``first modern paper on numerical analysis" \cite{b3}. The stability of Gaussian elimination under complete pivoting has been a major open problem ever since.
In 1961, Wilkinson produced a quasi-polynomial upper bound for the growth factor under complete pivoting of $2 \sqrt{n} n^{\tfrac{1}{4}\log(n)}$ and remarked that ``no matrix has been encountered for which [the growth factor for complete pivoting] was as large as 8'' \cite{b4}.
Bisain, Edelman, and Urschel recently improved this upper bound to approximately $n^{0.207 \log n}$ \cite{b5}. In the same decade that Wilkinson commented on the small growth factors encountered in practice,  a folklore conjecture that the growth factor under complete pivoting was at most $n$ arose, which ``became one of the most famous open problems in numerical analysis" \cite[pg. 181]{b1}.
This conjecture was later shown to be false in dimension $13$ in floating point arithmetic by Gould \cite{b6} (extended to exact arithmetic by Edelman \cite{b7}), and Edelman and Urschel recently showed it to be false for all $n \ge 11$ by a multiplicative constant \cite{b8}. Afterward, Fedorovskii produced examples of matrices with growth on the order of $n^{\log_{7/2}(4)}$ under complete pivoting \cite{b9}. However, the gap between lower and upper bounds for the growth factor under complete pivoting in the literature remained large.
We refer the reader to \cite{b1,b8,b10} for more details.

In contrast, partial pivoting, which takes the largest magnitude entry in the first column as the pivot, is completely understood in the worst case. In his 1965 text {\it The Algebraic Eigenvalue Problem} \cite{b11}, Wilkinson proved the upper bound of $2^{n-1}$ and produced an example matrix, commonly referred to as the Wilkinson matrix, that achieves it. Later, Higham and Higham characterized the full set of matrices with growth $2^{n-1}$ \cite{b12}. An exponentially large growth factor can be a prohibitive barrier to accurate computation in floating point arithmetic. Despite this, due to the low cost and memory access requirements of partial pivoting, Gaussian elimination with partial pivoting is the premier computational technique for solving a linear system and is the algorithm utilized by the built-in linear system solver in nearly every programming language. Much of the analysis and debate concerning partial pivoting regards whether large growth can occur in practice \cite{b13,b14,b15,b16}, and to what extent large growth is unstable (i.e., its possible resolution by the addition of a small random perturbation) \cite{b17}. We refer the reader to \cite{b18,b10} for more details.

Due to the exponential worst-case behavior of partial pivoting, a number of alternate pivoting strategies have been suggested. One of the most notable is rook pivoting, which takes the largest magnitude entry in its row and column as the pivot. Rook pivoting serves as a compromise between partial and complete pivoting: the strategy appears to be relatively fast in practice \cite{b19,b20}, like partial pivoting, and as Foster proved in \cite{b19}, it is known to have at most a quasi-polynomial growth factor $\frac{3}{2} \, n^{3\log(n)/4}$, like complete pivoting.
Also like complete pivoting, there is a substantial gap between the upper and lower bounds, with the largest known growth factor for rook pivoting being only of order $n^{1.669}$, proven by Edelman and Urschel \cite{b8}. Another alternative which has received recent attention is the use of a randomized version of partial pivoting, which chooses the pivot from the first column probabilistically, with weights depending on the relative magnitudes of the entries.

For the purposes of worst-case analysis, when we consider pivoting strategies such as partial, rook, or complete pivoting, we need not concern ourselves with the choice of permutation matrices during the algorithm, only with the final row- and column-permuted matrix. For this reason, it suffices to restrict ourselves to the set of matrices that require no pivoting under a given pivoting strategy, and analyze the growth factor without pivoting over those sets. To that end, we define the following three classes:
\begin{enumerate}
    \item $\CP_n(\R)$: $A\in\GL_n(\R)$ is completely pivoted if
\[\abs{a\at k_{11}}\ge\max_{i,j\in[n-k]} \abs{a\at k_{ij}} \qquad \text{for} \quad k = 0,\ldots,n-1,\]
    \item $\RP_n(\R)$: $A\in\GL_n(\R)$ is rook pivoted if
\[\abs{a\at k_{11}}\ge\max_{j\in[n-k]}\max\left\{\abs{a\at k_{1j}},\abs{a\at k_{j1}}\right\}  \qquad \text{for} \quad k = 0,\ldots,n-1,\]
    \item $\PP_n(\R)$: $A\in\GL_n(\R)$ is partially pivoted if
\[\abs{a\at k_{11}}\ge\max_{j\in[n-k]}\abs{a\at k_{j1}}  \qquad \text{for} \quad k = 0,\ldots,n-1,\]
\end{enumerate}
where $[n]:=\{1,\ldots,n\}$.
In addition, in this work we analyze two randomized pivoting strategies, which correspond to distributions on permutation matrices. Each permutation matrix $P$ can be identified with the permutation sequence $\pare{\pi_1,\ldots,\pi_n}$, where $\bm e_j^\top P \bm e_{\pi_j}=1$. Let $A[S,T]$ denote the submatrix of a matrix $A$ indexed by row indices $S$ and column indices $T$. The two randomized pivoting strategies we consider are volume pivoting, characterized by certain marginals
\eq{\label{a1}
P\sim\VolPP(A):\qquad \Pr(\set{\pi_1,\ldots,\pi_k}=S)\propto(\det A[S,[k]])^2,}
and randomized partial pivoting, for $p>0$, which is fully defined
\eq{\label{a2}P\sim\RandPP_p(A): \qquad \Pr(\pi_{k+1}=i\,|\,\set{\pi_1,\ldots,\pi_{k}}=S)\propto\abs{\frac{\det A[S+\set i,[k+1]]}{\det A[S,[k]]}}^p.}
It is not obvious from the definition of $\VolPP$ that a distribution with those prescribed marginals exists; indeed part of our contribution is showing that this is the case (\Cref{a3}).

\subsection{Our Contributions} Here we make major progress on a number of questions regarding the growth factor under various deterministic and randomized pivoting strategies.

First, we prove that the growth factor under partial pivoting can remain exponentially large even for sparse matrices. The worst-case bound of $2^{n-1}$ is achievable by matrices with as few as $4n-4$ non-zero entries, which is tight. Matrices that are locally sparse, i.e., with at most $k$ non-zero entries per row and column, can also have exponentially large growth under partial pivoting. For example, there exist matrices with at most three non-zero entries per row and column with growth at least $\varphi^{n-1}$, where $\varphi$ is the golden ratio. In contrast, any partially pivoted matrix with at most two non-zero entries per row and column has growth factor at most two.

\begin{theorem}[\Cref{a4} \& \Cref{a5} / Sparse exponential growth of partial pivoting]
For every $n\ge 2$, there is a matrix $A \in \PP_n(\R)$ with $\nnz(A) = 4n - 4$ and $\growth(A) = 2^{n-1}$, and this is the minimum possible value of $\nnz(A)$ for any matrix in $\PP_n(\R)$ with growth $2^{n-1}$. In addition, for all integers $2 \le k < n$, there is a matrix $A \in \PP_n(\R)$ with at most $k+1$ non-zero entries per row and column and $\growth(A) = r_k^{n-1} \ge (2-2^{1-k})^{n-1}$, where $r_k$ is the unique solution to $x^{k} = x^{k-1} + \ldots + x + 1$ in $(2-2^{1-k},2)$. Any matrix $A\in \PP_n(\R)$ with at most two non-zero entries per row and column has $\growth(A) \le 2$.
\end{theorem}

Despite this, we prove the existence of a randomized row pivoting strategy with polynomially bounded growth factor in expectation. In particular, we produce polynomial estimates for the Frobenius norms of $L$ and $U$, which translate immediately to polynomial estimates for growth. We compare these estimates to Frobenius norm lower bounds produced by Sylvester Hadamard matrices for any row permutation and lower bounds produced by arbitrary Hadamard matrices for any row and column permutation.

\begin{theorem}[\Cref{a6} / Existence of a low-growth row permutation]
For every $n \in \N$ and matrix $A\in\GL_n(\R)$, $\VolPP(A)$ is a well-defined probability distribution on permutations. If $P \sim\VolPP(A)$ then $PA$ has an LU factorization with probability one, $\E \|L\|^2_F \le (\frac{1}{6} + o(1))n^3$, $\E \|U\|^2_F \le (\frac{1}{6} + o(1))n^3 \|A\|_{1 \rightarrow 2}^2$, and
$\E\growth(P A)\le(\frac1{\sqrt6}+o(1))n^2$.
\end{theorem}

\begin{theorem}[\Cref{a7} \& \Cref{a8} / Lower bounds over all pivoting strategies]
Let $n = 2^k$, $H \in \GL_n(\R)$ be the Sylvester Hadamard matrix, and $P$ be any permutation matrix such that $PH$ has an LU factorization. Then $\|L\|_F^2 =n^{\log_2 3}$ and $\|U\|_F^2 = n^{\log_2 6}$. Furthermore, for any Hadamard matrix $A$ of any dimension $n >1$ and any permutation matrices $P$ and $Q$ such that $PAQ$ has an LU factorization, we have $\|L\|_F^2 \ge \frac{1}{e} n \log n$ and $\|U\|_F^2 \ge \frac{e}{2} n^2$.
\end{theorem}

While the existence of a row permutation with polynomial growth is an encouraging result, we also note that, in general, finding the row permutation that minimizes the growth factor is NP-complete. Furthermore, we show that randomized partial pivoting, the best candidate for a fast row pivoting strategy with polynomial growth factor, can have a growth factor that is nearly exponential.

\begin{theorem}[\Cref{a9} / Hardness of minimizing growth]\label{a10}
Let $\mathcal L$ be the language consisting of tuples $(A,g)\in\Mat_n(\Z)\times\frac12\Z_+$ for which there exists a permutation $P$ such that $\growth(PA)\le g$.
$\mathcal L$ is NP-complete.
\end{theorem}

\begin{theorem}[\Cref{a11} / Large growth of randomized partial pivoting]
For every $n$ and $p\in(0,\infty)$, there is a matrix $A\in\GL_n(\R)$ such that if $P \sim\mathrm{RandPP}_p(A)$, then $\growth(PA)\ge e^{\Omega_p(n/\log n)}$ with high probability.
\end{theorem}

Finally, we turn our attention to pivoting strategies that allow both row and column permutations. It was long thought that rook and complete pivoting were both likely to have a polynomial growth factor. Known upper bounds for both are currently quasi-polynomial, while the known lower bounds for both are sub-quadratic. Here we prove the existence of matrices for which rook and complete pivoting have quasi-polynomial growth.
\begin{theorem}[\Cref{a12} / Quasi-polynomial growth of rook pivoting]
For every $n$, there is a $A\in\RP_n(\R)$ with $\growth(A)\ge n^{(\frac14-o(1))\log_2 n}$.
\end{theorem}
\begin{theorem}[\Cref{a13} / Quasi-polynomial growth of complete pivoting]
For every $n$, there is a $A\in\CP_n(\R)$ with $\growth(A)\ge n^{(\frac1{16}-o(1))\log_2n}$.
\end{theorem}
This characterization of the asymptotic behavior of the growth factor under complete pivoting answers one of the oldest and most fundamental open questions in numerical analysis.

\section{Sparse Matrix Growth Under Partial Pivoting}

The growth factor under partial pivoting is known to be at most $2^{n-1}$, as the entries of successive iterates can only double in size at each of the $n-1$ steps of Gaussian elimination. Wilkinson showed that this upper bound is achieved by
\[A = \begin{pmatrix*}[r] 1 & 0\phantom{.} & \cdots & 0 & 1 \\ -1 & \ddots & \ddots & \vdots & \vdots \\
\vdots & \ddots & 1 &  0 & 1 \\
-1 & \cdots & -1 & 1 & 1 \\
-1 & \cdots & -1 & -1 & 1 \end{pmatrix*},\]
now commonly referred to as the Wilkinson matrix \cite{b11}. Later, Higham and Higham characterized the subset of $\PP_n(\R)$ that achieves the $2^{n-1}$ growth factor in the following theorem:

\begin{theorem}[{\cite[Theorem 2.2]{b12}}]\label{a14}
Every $A \in \PP_n(\R)$ with $\growth(A) = 2^{n-1}$ is of the form 
\[A = D M \begin{bmatrix} \widetilde T & \theta \bm d \end{bmatrix}, \qquad \text{where} \quad \widetilde T =  \begin{bmatrix} T  \\ 0 \end{bmatrix}, \]
$D \in \GL_n(\R)$ is a $\pm 1$ diagonal matrix, $M \in \GL_n(\R)$ is unit lower triangular with $M_{ij} = -1$ for $i>j$, $T \in \GL_{n-1}(\R)$ is an upper triangular matrix, $ \bm d= (1,2,4,\ldots,2^{n-1})^\top \in \R^n$, and $\theta$ is a scalar such that $\theta = |A_{1n}| = \maxnorm{A}$.
\end{theorem}
The exponential growth factor exhibited by both the Wilkinson matrix and the class of matrices described in the above theorem leads to catastrophic errors in floating point arithmetic \cite{b21,b10}. Within the numerical analysis community, there has been an ongoing discussion as to whether such matrices ever arise in practice. Examples of exponential growth coming from application were produced by Wright for two-point boundary-value problems \cite{b22} and by Foster for quadrature for Volterra integral equations \cite{b23}. One notable aspect of Wright's example is that, unlike the Wilkinson matrix, the matrices in question were sparse, with at most three non-zero entries per row and growth factor approaching $2^{n/2}$.

Here, we make progress in two directions. First, we exactly characterize how sparse a matrix can be and still attain growth $2^{n-1}$. Second, we prove lower bounds for how large the growth factor of a matrix with at most $k$ non-zero entries per row and column can be.

\begin{theorem}\label{a4}
Let $n \ge 2$ and $A \in \PP_n(\R)$ with $\growth(A) = 2^{n-1}$. Then $\nnz(A) \ge 4n -4$, with equality achieved by the matrix
\[A =  \begin{pmatrix*}[r]
     -\tfrac{1}{2} & \tfrac{1}{2}    &    &   & 1\\
    \tfrac{1}{2} &    -1 & \ddots &   & \vdots \, \\
    \vdots \, &      & \ddots & \tfrac{1}{2} & 1\\
    \vdots \, &      &   & -1 & 1\\
    \tfrac{1}{2} &      &   &   & 1
  \end{pmatrix*}.\]
\end{theorem}

\begin{proof}
The theorem statement is immediate when $n = 2$. Now, let $n\ge 3$ and consider the desired lower bound $\nnz(A) \ge 4n -4$ using the formulation of $A \in \PP_n(\R)$ with $\mathrm{growth}(A) = 2^{n-1}$ provided by \Cref{a14}. In \Cref{a14}, the matrix $D$ simply re-signs rows, and leaves the sparsity structure unchanged. Therefore, we may instead analyze $B:= M \begin{bmatrix} \widetilde T & \theta \bm d \end{bmatrix}$.

The first column of $B$ equals $T_{11} M\bm e_1$ for $T_{11} \ne 0$, and the last column equals $\theta \bm 1$. This already accounts for $2n$ non-zero entries in $B$. It suffices to show that the remaining $n-2$ columns all have at least two non-zero entries. Consider column $j$ for $1<j<n$. Because $\widetilde T$ is upper triangular,
\[\widetilde T_{j+1,j} = \bm e_{j+1}^\top M^{-1} B \bm e_j = 0,\]
where $\bm e_{j+1}^\top M^{-1} = (2^{j-1}, 2^{j-2}, \dots, 2 , 1 , 1 , 0 , \dots , 0 )$ is non-zero in its first $j+1$ entries. Therefore, either $\nnz(B\bm e_j) \ge 2$ or the first $j+1$ entries of $B\bm e_j$ are zero. But, by the non-singularity of the leading $j\times j$ submatrix of $M$, the first $j$ entries of $\widetilde T \bm{e}_j$ are also zero. In particular, $T_{jj} = 0$, a contradiction.

Finally, we verify that the example matrix $A$ with $\nnz(A) = 4n-4$ in the theorem statement is partially pivoted and has growth factor $2^{n-1}$. Indeed, the iterates of this matrix are given by
\[A\at k =  \begin{pmatrix*}[r]
     -\tfrac{1}{2} & \tfrac{1}{2}    &    &   & 2^k\\
    \tfrac{1}{2} &    -1 & \ddots &   & \vdots \; \, \\
    \vdots \, &      & \ddots & \tfrac{1}{2} & 2^{k}\\
    \vdots \, &      &   & -1 & 2^{k}\\
    \tfrac{1}{2} &      &   &   & 2^{k}
  \end{pmatrix*} \quad \text{for} \quad k \in [n-3], \qquad A^{(n-2)} = \begin{pmatrix*}[r] -\frac{1}{2} & 2^{n-2} \\[0.25 ex] \frac{1}{2} & 2^{n-2} \end{pmatrix*}, \qquad  A^{(n-1)} = \begin{pmatrix} 2^{n-1} \end{pmatrix}.\]
\end{proof}

\begin{theorem}\label{a5}
Let $k >1$, $n > k$, $r_k$ be the unique solution to $x^{k} = x^{k-1} + \ldots + x + 1$ in the interval $(2-2^{1-k},2)$, $s_j = \sum_{i = 1}^{k+1-j} r_k^{-i}$ for $j = 1,\ldots,k$, and $A \in \Mat_n(\R)$ be given by 
\[A_{ij} = \begin{cases} 1 & i = j,\, j<n \\
-1 &1\le i-j\le k \\s_i & j = n, \,i \le k \\ 0 & \text{otherwise} \end{cases}.\]
Then $A \in \PP_n(\R)$, has at most $k+1$ non-zero entries in each row and column, and \[\growth(A) = r_k^{n-1}\ge\pare{2-2^{1-k}}^{n-1}.\] In addition, any $A \in \PP_n(\R)$ with at most two non-zero entries in each row and column has $\growth(A) \le 2$.
\end{theorem}

\begin{example}
When $n = 6$, the matrices associated with $k = 2$ and $k=3$, respectively, in \Cref{a5} are
\[\begin{pmatrix} 1& & & & & 1 \\ 
-1& 1& & & & 1/\varphi \\ 
-1&-1 &1 & & &  \\ 
& -1&-1 & 1& & \\ 
& & -1& -1& 1& \\
& & &-1 & -1&  \end{pmatrix}\qquad \text{and} \qquad \begin{pmatrix} 1& & & & & 1 \\ 
-1& 1& & & & 1/\psi +1/\psi^2 \\ 
-1&-1 &1 & & & 1/\psi  \\ 
-1& -1&-1 & 1& & \\ 
& -1& -1& -1& 1& \\
& &-1 &-1 & -1&  \end{pmatrix},\]
where $\varphi = \frac{1+\sqrt{5}}{2}$ is the golden ratio and $\psi = \frac{1}{3}\big(1 + \sqrt[3]{19+3\sqrt{33}} + \sqrt[3]{19-3\sqrt{33}}\big)$.
\end{example}

\begin{proof}[Proof of \Cref{a5}]

First, we verify that the given matrix $A$, parameterized by $k$, has the desired properties. By inspection, $A$ has at most $k+1$ non-zero entries in each row and column, and, because $A$ is lower triangular save for the last column, with entries below the diagonal bounded in magnitude by the diagonal entries, $A$ is partially pivoted. To observe the growth factor, it suffices to give the LU factorization explicitly. Let $L$ be unit lower triangular with $-1$ on the first $k$ subdiagonals and zeros elsewhere and let $U$ be upper triangular with $U_{ii} =1$ for $i<n$, $U_{in} = r_k^{i-1}$ for all $i$ and zeros elsewhere. Then because $U$ is the identity matrix in its first $n-1$ columns, and the first $n-1$ columns of $L$ agree with the first $n-1$ columns of $A$, it suffices to analyze the last column of $A$. We have
\[(LU)_{in} = r_k^{i-1} - \sum_{j = 1}^{\min\{k,i-1\}} r_k^{i-1-j}.\]
For $i \le k$, using the identity $1 = r_k^{-1} + \ldots + r_k^{-k}$,
\[(LU)_{in} = r_k^{i-1} - \sum_{j=1}^{i-1} r_k^{i-1-j} = r_k^{i-1} \left( 1 - \sum_{j=1}^{i-1} r_k^{-j} \right) = r_k^{i-1}  \sum_{j=i}^{k} r_k^{-j} = \sum_{\ell = 1}^{k+1-i} r_k^{-\ell} = s_i,\]
and, for $i>k$,
\[(LU)_{in} = r_k^{i-1} - \sum_{j = 1}^{k} r_k^{i-1-j} = r_k^{i-1-k}\left( r_k^k - r_k^{k-1} - r_{k}^{k-2} - \dots - 1\right)=0.\]
In addition, by definition, $1 = s_1 > s_2 > \ldots > s_k = r_k^{-1}>0$, so $\growth(A) = r_k^{n-1}$.
That the root $r_k$ is unique and indeed in the interval $(2-2^{1-k},2)$ follows by Rouche's theorem on $(x^k-x^{k-1}-\cdots-1)\cdot\frac{x-1}{x^k}$ (see, e.g., \cite[Lemma 3.4]{b24}).

Now consider a matrix $A \in \PP_n(\R)$ with at most two non-zero entries per row and column. We prove the following claim by induction: 
\begin{quote}The matrix $A\at k$ has at most two non-zero entries per row and column, every row and column of $A\at k$ with two non-zero entries has infinity norm at most $\maxnorm{A}$, and every row and column of $A\at k$ with one non-zero entry has infinity norm at most $2 \maxnorm{A}$. \end{quote} 
The matrix $A$ certainly satisfies this claim. Now suppose $A\at k$ also satisfies this claim for some $k<n-1$. If the pivot is the only non-zero entry in either the first row or column,  $A^{(k+1)}$ is simply $A\at k$ with the first row and column removed, and therefore also satisfies the claim. If $a\at k_{i1} \ne 0$ and $a\at k_{1j} \ne 0$ for some $i$ and $j$, then, $A^{(k+1)}$ is simply $A\at k$ with the first row and column removed and only the entry $a\at k_{ij}$ updated. The update $-a\at k_{i1}a\at k_{1j}/a\at k_{11}$ is at most $\maxnorm{A}$ in magnitude, as, by the claim, $\big|a\at k_{1j}\big| \le \maxnorm{A}$, and, by $A \in \PP_n(\R)$, $\big|a\at k_{i1}\big| \le \big| a\at k_{11} \big|$. If $a\at k_{ij} = 0$, then $\big|a\at{k+1}_{i-1,j-1}\big| \le \maxnorm{A}$, and there are still at most two non-zero entries in its row and column, as the column containing $a\at k_{i1}$ and the row containing $a\at k_{1j}$ have been removed. If $a\at k_{ij} \ne 0$, then \[\big|a\at{k+1}_{i-1,j-1}\big| \le \big| a\at k_{ij} \big| + \abs{\frac{ a\at k_{i1}a\at k_{1j}}{a\at k_{11}}}  \le 2\maxnorm{A}.\] However, because $a\at k_{ij} \ne 0$, there are no other non-zero entries in the row or column of $a^{(k+1)}_{i-1,j-1}$. Therefore, $A\up{(k+1)}$ also satisfies the desired claim, thus proving the claim for all $k$. 

From here, it follows immediately that $\growth(A) \le 2$, as $A \in \PP_n(\R)$ implies that $\|L\|_{\max} =1$ and our proven claim implies that $\|U\|_{\max} \le 2 \|A\|_{\max}$.
\end{proof}

\section{Existence of Low-Growth Row Exchanges}
Partial pivoting exchanges rows so that the resulting lower triangular matrix $L$ has $\maxnorm L =1$. Unfortunately, this can result in the entries of $U$ becoming as large as $2^{n-1}$. This tradeoff is lessened by using complete pivoting, for which $\|L\|_{\max} = 1$  and the entries of $U$ grow at worst quasi-polynomially. However, complete pivoting requires row and column exchanges, and still does not produce a polynomial growth factor. Here, we consider a fundamental existence question: does there always exist a row permutation that produces polynomial growth? By analyzing the randomized pivoting strategy $P \sim \VolPP(A)$, we answer this question affirmatively.

This question is closely related to rank-revealing factorizations and maximum volume pivoting in numerical linear algebra (see \cite{b25,b26,b27,b28}). A $k \times k$ submatrix locally maximizes volume if its squared determinant cannot be increased by swapping out a row or column. This is naturally related to Gaussian elimination, as the entries of $L$ and $U$ are ratios of determinants:
\eq{L_{ij} = \frac{\det A [[j-1]+\set i,[j]]}{\det A[[j],[j]]} \quad \text{for} \quad i \ge j \qquad\text{and} \qquad U_{ij} =  \frac{\det A[[i],[i-1]\cup\set{j}]}{\det A[[i-1],[i-1]]} \quad \text{for} \quad i \le j.}

But there is a natural obstruction to applying these results directly. For any fixed $k$, one can find a large volume set of $k$ rows in the first $k$ columns of $A$, but these choices need not be nested as $k$ varies; a row pivoting strategy for Gaussian elimination requires a single chain of choices $S_1 \subset S_2 \subset \ldots \subset S_n$, where $|S_k| = k$. Using $\VolPP$ pivoting, we build a random chain with a determinantal distribution based on $\det(A[S_k,[k]])^2$, commonly referred to as volume sampling or a fixed size determinantal point process (DPP) \cite{b29}. Volume sampling has previously appeared in a variety of applied linear algebraic contexts, see, for instance, \cite{b30,b31,b32,b33,b34}.

The only structural result about DPPs that we require is the stochastic
monotonicity of projection processes.  If $\mathcal H_1\subseteq
\mathcal H_2$ are subspaces of $\ell^2(E)$ for some countable set $E$, and $\mathbb P^{\mathcal H_i}$
denotes the projection DPP associated with the orthogonal projection
onto $\mathcal H_i$, then
\(\mathbb P^{\mathcal H_1}\) is stochastically dominated by \(\mathbb P^{\mathcal H_2},\)
see \cite[Theorem 6.2]{b35}.  By Strassen's coupling characterization of stochastic domination \cite{b36}, the two
processes therefore admit a coupling $(S_1,S_2)$ such that
$S_1\subseteq S_2$ almost surely.  Applying this result successively to
the nested column spaces
\[
    \Col(A[:,[1]])
    \subseteq \Col(A[:,[2]])
    \subseteq\cdots\subseteq
    \Col(A)
\]
produces the required random nested chain, while preserving the
volume-sampling marginal at every level. 

\begin{theorem}\label{a6}
For every $A\in\GL_n(\R)$, there is a probability measure on permutations $\VolPP(A)$ with marginals \Cref{a1}, where if $P\sim\VolPP(A)$ then $PA=LU$ satisfies
\[\E\norm{L}_F^2\le\frac{n^3+5n}6\qquad\text{and} \qquad \E\norm U_F^2\le \frac{n^3+3n^2+2n}6\magn{A}_{1\to2}^2.\]
Since $\maxnorm{\cdot}\le\norm\cdot_F$ and $\norm A_{1\to2}\le\sqrt n\maxnorm A$, this additionally implies that $\E\growth(PA)^2\le\frac{n^4}6+O(n^3)$.
\end{theorem}
\begin{proof}
The existence of the distribution $\VolPP$ with the desired marginals \Cref{a1} is given by \Cref{a3}. The expected squared Euclidean norms for each column of $L$ and row of $U$ are bounded by \Cref{a15}, so we simply sum these expressions,
\spliteq{}{
\E\magn L_F^2&=\sum_{k=1}^n\E\magn{L[:,k]}^2\le n+\sum_{k=1}^nk(n-k)=\frac{n^3+5n}6,\\
\E\magn U_F^2&=\sum_{k=1}^{n} \E\magn{U[k,:]}^2\le\magn{A}^2_{1\to2}\sum_{k=1}^n k(n-k+1)=\frac{n^3+3n^2+2n}6 \magn{A}^2_{1\to2}.}
\end{proof}

\begin{lemma}
\label{a3}
Given $A\in\GL_n(\R)$, let $\Pi_k$ be the orthogonal projection onto the span of the leading $k$ columns of $A$.
Then there is a distribution, which we denote $\VolPP$, on permutations $(\pi_1,\ldots,\pi_n)$ of $[n]$
such that
\eq{\label{a16}\Pr( \set{\pi_1,\ldots,\pi_k}=S )=\det\Pi_k[S,S].}
for all $k\in[n]$. Furthermore, these distributions are equivalently formulated as
\[\Pr( \set{\pi_1,\ldots,\pi_k}=S )\propto(\det A[S,[k]])^2.\]
\end{lemma}
\begin{proof}
The equivalence of the formulas comes from the orthogonal projection formula, \[\Pi_k=A[:, [k]](A[:, [k]]^\top A[:, [k]])^{-1}A[:, [k]]^\top.\] Looking at the $[S,S]$ submatrix of both sides and taking determinants gives
\[
\det\pare{\Pi_k[S, S]}
=\frac
{( \det A[S, [k]] )^2}
{\det(A[:, [k]]^\top A[:, [k]])},\]
which indeed defines a probability distribution by Cauchy-Binet.
For each $k$, let $S_k$ be the random element of $\binom{[n]}k$ such that
\[\Pr(S_k=S)=\det\Pi_k[S,S].\]
By \cite[Theorem 6.2]{b35}, $S_{k+1}$ stochastically dominates $S_k$, and, by Strassen's coupling theorem \cite{b36}, this is equivalent to the existence of a coupling such that $S_k\subset S_{k+1}$. By applying this theorem pairwise, we have an entire nested sequence
\[\emptyset=S_0\subset S_1\subset S_2\subset\cdots\subset S_n=[n],\qquad \abs{S_k}=k.\]
Set $\pi_i$ to be the unique element in $S_i\backslash S_{i-1}$ to obtain the desired distribution $\VolPP(A)$ on permutations.

\end{proof}
\begin{lemma}\label{a15}
Let $A \in \GL_n(\R)$ and $P$ be a permutation sampled according to $\VolPP(A)$ of \Cref{a3}. Then $PA$ admits an LU factorization, $PA=LU$, where $L$ is unit lower triangular and $U$ is upper triangular, and
\[\E\norm{L[:,k]}^2\le1+k(n-k) \qquad \text{and} \qquad \E\magn{U[k,:]}^2\le k\sum_{i=k}^n\norm{Ae_i}^2\]
for all $k\in[n]$.
\end{lemma}
\begin{proof}
Set $C=[k]$. 
In order for an LU factorization to exist, one needs each leading pivot block to be invertible. That is, if $S$ is the set of the first $k$ pivots, then we need $\det A[S,C]\neq 0$. From the definition of $\VolPP(A)$, we see that if $\det A[S,C]=0$, $S$ will be the set of first $k$ pivots with probability 0; i.e. $\det A[S,C]\neq0$ for the $S$ sampled by $\VolPP(A)$ occurs with probability 1.

Let $(\pi_1,\ldots,\pi_n)$ be the permutation from \Cref{a3} and set $S=\set{\pi_1,\ldots,\pi_k}$, $S'=\set{\pi_1,\ldots,\pi_{k+1}}$. By Cramer's rule,
\[
\abs{L_{ik}}=
\abs{\frac
{\det A[S-\set{\pi_k}+\set{\pi_i},C]}
{\det A[S,C]}},\qquad
\abs{U_{k+1,j}}=
\abs{\frac
{\det A[S',C+\set j]}
{\det A[S,C]}}
.\]
We first handle $L$. We have
\spliteq{}{
\norm{ L[:,k] }^2
  =1+\sum_{i=k+1}^n\abs{L_{ik}}^2
  &=1+\sum_{i=k+1}^n\abs{\frac{ \det A[S-\set{\pi_k}+\set{\pi_i},C] }{ \det A[S,C] }}^2
\\&\le 1+\sum_{\pi\in S}\sum_{\pi'\not\in S}\abs{\frac{ \det A[S-\set{\pi}+\set{\pi'},C] }{ \det A[S,C] }}^2.}
Taking expectations with $\Pr(S)\propto(\det A[S,C])^2$, the denominator cancels, giving
\[
\E\norm{ L[:,k] }^2
\le 1+\frac{ \sum_{S}  \sum_{\pi \in S}\sum_{\pi'\not\in S}(\det A[S-\set{\pi}+\set{\pi'},C])^2 }{ \sum_{S}(\det A[S,C])^2 },
\]
where the outer sums are over $S\in\binom{[n]}k$. Note that $S-\set \pi +\set{\pi'}$ is just another element of $\binom{[n]}k$, and it is counted exactly $k(n-k)$ times, which gives $\E \norm{ L[:,k] }^2\le1+k(n-k)$. We handle $U$ in a similar fashion:
\spliteq{}{
\magn{U[k+1,:]}^2
=   \sum_{j=k+1}^n\abs{U_{k+1,j}}^2
&=\sum_{j=k+1}^n\abs{\frac{\det A[S',C+\set j]}{\det A[S,C]}}^2
\\&\le
\sum_{j = k+1}^n \sum_{\pi'\not\in S}\abs{\frac
{\det A[S+\set{\pi'},C+\set j]}
{\det A[S,C]}}^2
}
Once again, taking expectations with $\Pr(S)\propto(\det A[S,C])^2$ causes the denominator to cancel giving us
\spliteq{}{
\E \magn{U[k+1,:]}^2
 &\le\sum_{j = k+1}^n\frac{\sum_S\sum_{\pi'\not\in S}\abs
{\det A[S+\set{\pi'},C+\set j]}^2}{\sum_S\abs{\det A[S,C]}^2},
}
where the outer sum is over $S\in\binom{[n]}k$. Note that $S+\set{\pi'}$ is an element of $\binom{[n]}{k+1}$, and it is counted exactly $k+1$ times in the sum. Applying Cauchy-Binet to the numerator and denominator gives, for $B=A^\top A$,
\[\E \magn{U[k+1,:]}^2\le(k+1)\sum_{j = k+1}^n \frac{\det B[C+\set j,C+\set j]}{\det B[C,C]}\le(k+1)\sum_{j=k+1}^n B_{j,j},\]
where the final step used the bound $\det B[C+\set j,C+\set j]\le\det(B[C,C])\cdot B_{j,j}$, a consequence of Hadamard's maximum determinant inequality.
Replacing $k$ with $k-1$ gives the stated result.
\end{proof}

Here, we provide a lower bound as well, giving some indication as to how close \cref{a6} is to being tight. Whenever an $n\times n$ Hadamard matrix exists, it gives a lower bound of $n$ for the growth factor under any row or column pivoting strategy. Indeed, for any Hadamard matrix $H$ (note any row or column permutation of a Hadamard matrix is also Hadamard) with an LU decomposition $H = LU$, we have $|U_{n,n}| = n$ \cite{b12}. In order to produce lower bounds for $\|L\|_F$ and $\|U\|_F$ under row permutations, we utilize Sylvester Hadamard matrices.

\begin{theorem}\label{a7}
Let $k \in \mathbb{N}$, $H_k = \begin{pmatrix} 1 & 1 \\ 1 & -1 \end{pmatrix}^{\otimes k}$ be the $n = 2^k$ dimensional Sylvester Hadamard matrix, and $P$ be any permutation matrix for which $PH_k$ has an LU factorization $PH_k = LU$. Then
\[\|L\|_F^2 = 3^k = n^{\log_2 3} \qquad \text{and} \qquad \|U\|_F^2 = 6^k = n^{\log_2 6}.\]
\end{theorem}

\begin{proof}
We have
\[\begin{pmatrix} 1 & 1 \\ 1 & -1 \end{pmatrix} = \begin{pmatrix} 1 & 0 \\ 1 & 1 \end{pmatrix} \begin{pmatrix} 1 & 1 \\ 0 & -2 \end{pmatrix} \]
and so, by the Kronecker mixed product property,
\[H_k = \begin{pmatrix} 1 & 1 \\ 1 & -1 \end{pmatrix}^{\otimes k} = \left[ \begin{pmatrix} 1 & 0 \\ 1 & 1 \end{pmatrix} \begin{pmatrix} 1 & 1 \\ 0 & -2 \end{pmatrix} \right]^{\otimes k} =  \begin{pmatrix} 1 & 0 \\ 1 & 1 \end{pmatrix}^{\otimes k} \begin{pmatrix} 1 & 1 \\ 0 & -2 \end{pmatrix}^{\otimes k} = LU. \]
This completes the proof for $H_k$. 

What remains is to show that $PH_k$, if it has an LU factorization, has exactly the same Frobenius norms for $L$ and $U$. We proceed by induction on $k$ for the larger class of signed and row-permuted Sylvester matrices $A = D P H_k$, where $D$ is a $\pm 1$ diagonal matrix. We induct on the identities $\|L\|_F^2 = \|L^{-1}\|_F^2 = 3^k$. This immediately implies our desired result, as $A A^\top = 2^k I$, so $U U^\top = 2^k L^{-1} L^{-T}$ and $\|U\|_F^2 = 2^k \|L^{-1}\|_F^2$.

The base case of $k = 0$ is immediate. Now, let $n = 2^{k-1}$. The rows of $H_k$ consist of the $n$ pairs $(\bm e_i^\top H_{k-1},\bm e_i^\top H_{k-1})$ and $(\bm e_i^\top H_{k-1},-\bm e_i^\top H_{k-1})$ for $i\in [n]$. The first $n$ rows of $A$ consist of precisely one row from each of these pairs and possibly re-signs them, and the last $n$ rows consist of the other row from each pair, otherwise the leading $n\times n$ principal minor of $A$ would be zero. Therefore $A$ takes the form
\[A = \bmat{ \wt H & \wt D \wt H \\ \widehat H & \widehat D \widehat H},\]
where $\wt H$ and $\widehat H$ are signed row permutations of $H_{k-1}$, and $\wt D$ and $\widehat D$ are $\pm 1$ diagonal matrices. Let $Q$ be the signed permutation matrix that reorders and re-signs $\wt H$ to $\widehat H$, i.e., $\widehat H =  Q \wt H$. Then $Q \wt D Q^\top = - \widehat D$, and so $Q \wt D \wt H = - \widehat D \widehat H$. We have
\[A = \bmat{ \wt H & \wt D \wt H \\ \widehat H & \widehat D \widehat H} = \bmat{ I & 0 \\ Q & I } \bmat{\wt H & \wt D \wt H \\ 0 & 2 \widehat D \widehat H}.\]
Because $A$ has an LU factorization and the leading principal minors of $A$ equal the leading principal minors of the second matrix in the above decomposition, both $\wt H $ and $ \widehat D \widehat H$ have LU factorizations, which we induct on. Let $\wt H = \wt L \wt U$ and $ \widehat D \widehat H = \widehat L \widehat U$. Then $A$ has the LU factorization
\[\bmat{ \wt H & \wt D \wt H \\ \widehat H & \widehat D \widehat H} = \bmat{\wt L & 0 \\ Q \wt L & \widehat L } \bmat{\wt U & \wt L^{-1} \wt D \wt H \\ 0 & \widehat U - \widehat L^{-1} Q \wt D \wt H}\]
and
\[ L^{-1} =  \bmat{\wt L^{-1} & 0 \\ -\widehat L^{-1} Q  & \widehat L^{-1}} .\]
The matrix $Q$ is orthogonal, and so, by induction,
\[\|L\|_F^2 = \|\wt L\|_F^2 + \|Q \wt L\|_F^2 + \| \widehat L \|_F^2 = 2\|\wt L\|_F^2+ \| \widehat L \|_F^2 = 3 \times 3^{k-1} = 3^k\]
and
\[\|L^{-1}\|_F^2 = \|\wt L^{-1}\|_F^2 + \|\widehat L^{-1} Q\|_F^2 + \| \widehat L^{-1} \|_F^2 = \|\wt L^{-1}\|_F^2+ 2\| \widehat L^{-1} \|_F^2 = 3 \times 3^{k-1} = 3^k.\]
\end{proof}

When both rows and columns are reordered, the uniformity of the Frobenius norms of $L$ and $U$ no longer persists for Sylvester Hadamard matrices. However, we can produce a lower bound for any Hadamard matrix, thus providing a lower bound for any Hadamard matrix under row and column pivoting (as this class is closed under both operations).

\begin{theorem}\label{a8}
Let $n \ge 2$ and $H$ be an $n \times n$ Hadamard matrix with an LU factorization $H = LU$. Then
\[\|L\|_F^2 \ge n \sum_{k = 1}^n \frac{(k-1)^{k-1}}{k^k} \ge \frac{n \log n}{e} \qquad \text{and} \qquad \|U\|_F^2 \ge (n-1) + \sum_{k = 1}^n \frac{k^k}{(k-1)^{k-1}} \ge \frac{e n^2}{2} .\]
\end{theorem}

\begin{proof}
Let $\Delta_0^2:=1$ and $\Delta_k^2 = \prod_{i = 1}^k U_{ii}^2$, $k = 1,\ldots,n$, be the squared determinant of the leading $k \times k$ principal submatrix of $U$. By summing the squares of the diagonal entries and the $n-1$ off-diagonal entries in the first row of $U$, we have
\[\|U\|_F^2 \ge (n-1) + \sum_{k = 1}^n \frac{\Delta^2_k}{\Delta^2_{k-1}}\]
and, summing the squares of the diagonal entries of $U^{-1}$,
\[\|U^{-1}\|_F^2 \ge \sum_{k = 1}^n \frac{\Delta^2_{k-1}}{\Delta^2_k}.\]
By Hadamard's inequality, and the fact that $H$ is a Hadamard matrix, $\Delta^2_k \le k^k$ for $k \in [n-1]$, and $\Delta_n^2 = n^{n}$. Now, define $\bm x_k=\log \Delta_k^2$. We have $\bm x_0 = 0$, $\bm x_n = n \log n$, and $0\le \bm x_k \le k \log k$ for $k \in [n-1]$. Consider the functions
\[f(\bm x) = \sum_{k = 1}^{n} e^{x_k-x_{k-1}} \qquad \text{and} \qquad g(\bm x) = \sum_{k = 1}^{n} e^{-(x_k-x_{k-1})}.\]
Both are the sum of exponentials of linear functions, and so are convex functions over the convex set $0 \le \bm x_k \le k \log k$ for $k \in[n-1]$. Let $\bm x^*$ be the corner point $\bm x_k = k \log k$ for $k \in [n-1]$. We have
\[\frac{\partial f(\bm x^*)}{\partial \bm x_k} = e^{\bm x_k - \bm x_{k-1}} - e^{\bm x_{k+1}- \bm x_k} = \frac{k^k}{(k-1)^{k-1}}- \frac{(k+1)^{k+1}}{k^k} <0\]
and
\[\frac{\partial g(\bm x^*)}{\partial \bm x_k} = -e^{-(\bm x_k-\bm x_{k-1})} + e^{-(\bm x_{k+1}-\bm x_k)} = - \frac{(k-1)^{k-1}}{k^k} + \frac{k^k}{(k+1)^{k+1}} < 0.\]
Therefore, $\nabla f(\bm x^*)^\top(\bm y - \bm x^*) \ge 0$ and $\nabla g(\bm x^*)^\top(\bm y - \bm x^*) \ge 0$ for all feasible $\bm y$, and so $\bm x^*$ is the minimizer of both $f$ and $g$ over the convex set. This gives the lower bounds
\[\|U\|_F^2 \ge (n-1) + \sum_{k = 1}^n \frac{k^k}{(k-1)^{k-1}} \qquad \text{and} \qquad \|U^{-1}\|_F^2 \ge \sum_{k = 1}^n \frac{(k-1)^{k-1}}{k^k}.\]
Finally, because $H$ is a Hadamard matrix,
$\|L\|_F^2 = \|H U^{-1}\|_F^2 = n \|U^{-1}\|_F^2$.
\end{proof}

\section{Hardness of Optimal Row Exchanges}
The majority of research regarding pivoting in Gaussian elimination involves fixing a pivot strategy and analyzing the resulting element growth. One notable exception is the theorem \cite[Theorem 2.1]{b12} of Higham and Higham, which provides a lower bound for growth over all row and column permutations. Using this theorem, they provided numerous examples of matrices that have $\growth(PAQ)\ge n/2$ for all permutation matrices $P$ and $Q$ \cite[Equations 2.1 to 2.6]{b12}. Here, we show that recognizing matrices that admit a bounded-growth row ordering is NP-complete, even for the fixed threshold of growth equal to one.

Several related elimination-ordering problems are known to be computationally hard. For instance, computing a minimum-fill ordering for sparse symmetric elimination is NP-hard \cite{b37}, determining the pivots selected by partial or complete pivoting is P-complete \cite{b38}, recognizing matrices admitting a perfect partial elimination scheme is NP-hard \cite{b39}, and global maximum volume column selection, which underlies maximum-volume and rank-revealing pivoting, is NP-hard \cite{b40}.

Our reduction uses incidence matrices between variables and clauses. Let $i \in [n]$ index variables, $j \in[m]$ index clauses, and, for an assignment $\bm b\in\set{0,1}^n$, let  \(M_{\bm b}\in\{0,1\}^{n\times m}\) be such that \((M_{\bm b})_{ij}=1\) when setting \(x_i= \bm b_i\) implies that the $j$th clause is true. 
The \(j\toth\) coordinate of \(\one^\top M_{\bm b}\) is the number of unique literals of the $j^{th}$ clause satisfied by \(\bm b\). The use of such a matrix in a hardness reduction is not novel. Our new ingredient here is to place \(M_{\bm b}\) inside a block Gaussian elimination gadget so that elimination produces the row
\[
    \frac32\one^\top-\frac12\one^\top M_{\bm b}.
\]
The key property of this vector is that its \(j\toth\) entry is at most \(1\) exactly when the $j\toth$ clause is satisfied.

\begin{definition}
    Let $f$ be a 3-CNF formula with variables $x_i$ indexed by $[n]$ and clauses $c_j$ indexed by $[m]$. For an assignment $\bm b\in\set{0,1}^n$, let $M_{\bm b}$ be the matrix
    $(M_{\bm b})_{ij}=[[(x_i= \bm b_i)\implies c_j]]$ and define
    \[A_f = \, \begin{blockarray}{cccccc}
         {\scriptstyle 1} & {\scriptstyle n} & {\scriptstyle 1} & {\scriptstyle n} & {\scriptstyle m} & \\[-.25 ex]
       \begin{block}{[ccccc]c}
       & & & & & \\[-1.5 ex]
         1  &  & & & -\one^\top &{ \scriptstyle 1} \\[.25 ex] 
  & I & & \phantom{-}\frac12 I & M_\zero &{ \scriptstyle n} \\[.25 ex]  
   & I & & - \frac12 I & M_\one &{ \scriptstyle n} \\[.25 ex]  
   \frac12  & \frac12 \one^\top & 1 & & \one^\top &{ \scriptstyle 1} \\[.25 ex]  
      &  & & & I &{ \scriptstyle m} \\[.25 ex] 
    \end{block}
     \end{blockarray} \in \Mat_{2n+m+2}(\tfrac12 \Z).\]
    Notice that $A_f$ has dimension $2n+m+2$ and $m$ and can be constructed from $f$ in polynomial time.
\end{definition}

\begin{theorem}\label{a9}
If $f$ is satisfiable, then there is a permutation $P$ such that $\growth(PA_f)=1$. If $f$ is not satisfiable, then for every permutation $P$, one has $\growth(PA_f)\ge\frac32$.
\end{theorem}
\begin{proof}
Suppose that $f$ is satisfiable. We construct a permutation $P$ such that $\growth(PA_f)=1$. First note that $\maxnorm{A_f}=1$, so it suffices to bound $\maxnorm L$ and $\maxnorm U$ by one.
For an assignment $\bm b \in\set{0,1}^n$, let $P_{\bm b}$ be the associated permutation that swaps rows $1+i$ and $1+i+n$ if and only if $\bm b_i=1$. Then
\[P_{\bm b} A_f = \, \begin{blockarray}{cccccc}
         {\scriptstyle 1} & {\scriptstyle n} & {\scriptstyle 1} & {\scriptstyle n} & {\scriptstyle m} & \\[-.25 ex]
       \begin{block}{[ccccc]c}
       & & & & & \\[-1.5 ex]
         1  &  & & & -\one^\top &{ \scriptstyle 1} \\[.25 ex] 
  & I & & \phantom{-}\frac12 D & M_{\bm b} &{ \scriptstyle n} \\[.25 ex]  
   & I & & - \frac12 D & M_{\one - \bm b} &{ \scriptstyle n} \\[.25 ex]  
   \frac12  & \frac12 \one^\top & 1 & & \one^\top &{ \scriptstyle 1} \\[.25 ex]  
      &  & & & I &{ \scriptstyle m} \\[.25 ex] 
    \end{block}
     \end{blockarray}, \]
     where $D = I - 2 \diag(\bm b)$ is a diagonal $\pm 1$ matrix. We apply one further permutation, which swaps the third and fourth block rows, to form $\widetilde P_{\bm b}$:
     \[\widetilde P_{\bm b} A_f = \, \begin{blockarray}{cccccc}
         {\scriptstyle 1} & {\scriptstyle n} & {\scriptstyle 1} & {\scriptstyle n} & {\scriptstyle m} & \\[-.25 ex]
       \begin{block}{[ccccc]c}
       & & & & & \\[-1.5 ex]
         1  &  & & & -\one^\top &{ \scriptstyle 1} \\[.25 ex] 
  & I & & \phantom{-}\frac12 D & M_{\bm b} &{ \scriptstyle n} \\[.25 ex]  
   \frac12  & \frac12 \one^\top & 1 & & \one^\top &{ \scriptstyle 1} \\[.25 ex]  
      & I & & - \frac12 D & M_{\one - \bm b} &{ \scriptstyle n} \\[.25 ex]  
      &  & & & I &{ \scriptstyle m} \\[.25 ex] 
    \end{block}
     \end{blockarray}.\]
  By inspection, $\widetilde P_{\bm b} A_f = L U$, where
  \[L = \, \begin{blockarray}{cccccc}
         {\scriptstyle 1} & {\scriptstyle n} & {\scriptstyle 1} & {\scriptstyle n} & {\scriptstyle m} & \\[-.25 ex]
       \begin{block}{[ccccc]c}
       & & & & & \\[-1.5 ex]
         1  &  & & &  &{ \scriptstyle 1} \\[.25 ex] 
  & I & &  &  &{ \scriptstyle n} \\[.25 ex]  
   \frac12  & \frac12 \one^\top & 1 & &  &{ \scriptstyle 1} \\[.25 ex]  
      & I & & I & &{ \scriptstyle n} \\[.25 ex]  
      &  & & & I &{ \scriptstyle m} \\[.25 ex] 
    \end{block}
     \end{blockarray} \qquad \text{and} \qquad U = \, \begin{blockarray}{cccccc}
         {\scriptstyle 1} & {\scriptstyle n} & {\scriptstyle 1} & {\scriptstyle n} & {\scriptstyle m} & \\[-.25 ex]
       \begin{block}{[ccccc]c}
       & & & & & \\[-1.5 ex]
         1  &  & & & -\one^\top &{ \scriptstyle 1} \\[.25 ex] 
  & I & & \phantom{-}\frac12 D & M_{\bm b} &{ \scriptstyle n} \\[.25 ex]  
     & & 1 & -\frac14 \one^\top D & \frac32 \one^\top - \frac12 \one^\top M_{\bm b} &{ \scriptstyle 1} \\[.25 ex]  
      &  & & -  D & M_{\one - \bm b} - M_{\bm b} &{ \scriptstyle n} \\[.25 ex]  
      &  & & & I &{ \scriptstyle m} \\[.25 ex] 
    \end{block}
     \end{blockarray} .\]
     We have $\maxnorm L = 1$ and $\maxnorm U = \max \{1,\norm{\frac32\one-\frac12M_{\bm b}^\top\one}_\infty\}$. If $\bm b$ is indeed a satisfying assignment, then $M_{\bm b}$ has $1,2,$ or $3$ non-zero entries in each column, all equal to $1$, and so $M_{\bm b}^\top\one\in\set{1,2,3}^m$. This means $\maxnorm U = 1$ and $\growth(\wt P_{\bm b} A_f ) = 1$, as desired.

Conversely, suppose that there exists a permutation $P$ with $\growth(PA_f)<\frac32$. The first pivot must come from the first block, as using the fourth block leads to $\maxnorm{L} \ge 2$. Now, consider the next $n$ pivots chosen. Again, during these $n$ steps,  using the fourth block leads to $\maxnorm{L} \ge 2$, so the $(1+i)\toth$ pivot must have index $1+i$ or $1+i+n$, corresponding to the choice $\bm b_i =0$ or $\bm b_i = 1$. In particular, the first $n+1$ rows of the permutation $P$ must agree with a permutation $P_{\bm b}$ for some $\bm b$. There is only one choice for the $(n+2)^{nd}$ pivot, the entry in the third block row, and so the third block row and fifth block column of the resulting $U$ must contain $\frac32 \one^\top - \frac12 \one^\top M_{\bm b}$. In order for $\growth(PA_f)<\frac32$, we must have $\norm{\frac32\one-\frac12M_{\bm b}^\top\one}_\infty<\frac32$. In particular, every entry of $M_{\bm b}^\top\one$ must be positive, which means that $\bm b$ is a satisfying assignment of $f$.
\end{proof}

Theorem \ref{a9} immediately implies our desired Theorem \ref{a10}. Scaling $A_f$ by a factor of two places all entries in $\Z$ and preserves the growth factor of it and all its permutations. Additionally, the factorization provided in the above proof is valid for any choice of $\bm{b}$, independent of whether it satisfies $f$ or not, and provides a certification that $A_f$ is non-singular, as the diagonal blocks $1$, $I$, $1$, $-D$, and $I$ of $U$ are all non-singular. To see that this language is a member of NP, we note that the computation of a minor of a matrix can be performed in polynomial time, and its bit length is bounded. In particular, if the entries of $A \in \Mat_n(\R)$ have bit length at most $\ell$, then, by Hadamard's inequality, every minor of $A$ has bit length $O(n(\ell + \log n))$. The existence of an LU factorization of a matrix and the entries of the LU factorization, if it exists, are therefore computable in polynomial time.

\section{Large Growth of Randomized Partial Pivoting}
Recall that randomized partial pivoting, $\RandPP_p$, is a procedure for performing row exchanges where the next pivot index $\pi_{k+1}$ is selected according to the marginal distribution
\eq{\label{a17} \Pr(\pi_{k+1}=i\,|\,\set{\pi_1,\ldots,\pi_{k}}=S)\propto\abs{\frac{\det A[S+\set i,[k+1]]}{\det A[S,[k]]}}^p,}
where $\pi_1,\ldots,\pi_{k}$ are the pivot rows chosen at steps $1$ through $k$ of the algorithm and $p \in (0,\infty)$. By \Cref{a0},
\[A\at{k} = A[S^c,[k]^c]- A[S^c,[k]]A[S,[k]]^{-1} A[S,[k]^c],\]
and so, indexing the rows and columns of $A\at k$ using the indices of $S^c$ and $[k]^c$, respectively,
\[a\at k_{ij} = a_{ij} - A[\set i,[k]] A[S,[k]]^{-1} A[S,\set j]. \]
Therefore, by Schur's formula,
\[\Pr(\pi_{k+1}=i\,|\,A\at k)\propto\abs{\frac{\det A[S+\set i,[k+1]]}{\det A[S,[k]]}}^p = \abs{a\at k_{i,k+1}}^p. \]
When $p\to\infty$, one recovers ordinary deterministic partial pivoting (up to random selection among ties); when $p\to0$, one obtains uniform pivoting. Neither extreme is desirable: partial pivoting has exponential growth in $U$ in the worst case, and uniform pivoting is liable to select tiny pivots, leading to arbitrarily large growth in $L$. For this reason, the numerical analysis community has focused on intermediate values of $p$, most naturally $p=1$ and $p=2$, which favors large pivots, and, hence, usually moderate multipliers, while also occasionally deviating from the greedy trajectory.

A folklore expectation that followed naturally was that this process should introduce enough randomness to ensure polynomial growth with high probability. We are not aware of a formal statement of this conjecture in the published literature, but similar ideas have appeared. The complete-pivoting analogue to $\RandPP_p$ is to select both a row and column index, in proportion to a power of the entry, $\propto\sabs{a_{ij}\at k}^p$; the case $p=2$ was introduced in \cite{b41}. In addition, there is a more specialized analogue in the positive-definite setting, where one samples the row and column indices to be the same, in proportion to a power of the diagonal entry, $\propto\sabs{a_{ii}\at k}^p$ \cite{b42}. 

Separately, randomization is known to yield results elsewhere in the elimination process.  Average-case and
smoothed analyses randomize the input matrix and establish polynomial
growth or precision bounds for Gaussian inputs and Gaussian
perturbations
\cite{b13,b43,b44,b16}.  Randomized complete pivoting instead uses a random
sketch of the Schur complement to locate an approximately maximal
pivot, obtaining complete pivoting-type growth bounds with high
probability \cite{b45}.  In both cases, the source and
role of the randomness differ from that of \(\RandPP_p\), though they highlight the apparent utility of randomization.

Here, we disprove any such conjecture that sampling $P\sim\RandPP_p(A)$ results in low growth with high probability. In particular, we provide an example for which the growth is nearly exponential.
Our construction is the $Q$ factor from the QR decomposition of the following ill-conditioned transposed Jordan block
\eq{\label{a18}N=\pmat{z\\1&z\\&1&z\\&&\ddots&\ddots\\&&&1&z}=QR}
for $z=\exp\big(-\frac1{pn^\alpha}\big)\approx1-\frac1{pn^\alpha}$, for a tunable $\alpha\in(0,1)$ which is allowed to depend on $n$.
Our strategy is to show that the final pivot chosen by $\RandPP_p(Q)$ concentrates toward the end of the matrix. That is, if $\pi_1,\ldots,\pi_n$ is the random sequence of pivot indices chosen by $\RandPP_p(Q)$, then $\pi_n$ will concentrate on the set $[n-o(n),n]$. Under this event, the growth of $Q$ under $\RandPP_p$ pivoting will resemble the growth of $Q$ with no pivoting up to $n-o(n)$ iterations, which is exponential in $z$.

\begin{lemma}\label{a19}
    If $(\pi_1,\ldots,\pi_n)\sim\RandPP_p(Q)$, then $\Pr(\pi_n\ge n-n^\alpha)\ge1-ne^{-n^\alpha/(e+1)}$.
\end{lemma}
\begin{proof}
The central insight is that the expression for the distribution of the next pivot index from \Cref{a17} is invariant under multiplication of $A$ on the right by non-singular upper triangular matrices. In particular, 
\[\det\pare{Q[S,:] R[:,[k]]}  = \det \pare{Q[S,[k]]} \det\pare{R[[k],[k]]},\]
and so
\[\RandPP_p(N)\overset{d}{=}\RandPP_p(Q).\]
This is quite convenient, as $\RandPP_p(N)$ is easier to control. 

Note that the first $k$ columns of $N$ only have non-zero entries in the first $k+1$ rows, and so the first $k$ pivot indices are contained in $\set{\pi_1,\ldots,\pi_k}\subset[k+1]$. The set $\set{\pi_1,\ldots,\pi_k}$ has exactly one element missing from $[k+1]$. For our purposes, it is convenient to think not of sets of pivot indices, but of the sequence of row indices that are \textit{not} pivots: let $m_{k+1}$ be the unique element of $[k+1]-\set{\pi_1,\ldots,\pi_k}$. Note that $\pi_n=m_n$. 

The relevant submatrices have a particularly nice block diagonal structure:
\[
N\left[[k+1]-\set{m_{k+1}},[k]\right]
=
\left(
\begin{array}{c|c}
\begin{matrix}
z      &        &        \\
1      & \ddots &        \\
       & \ddots & z
\end{matrix}
&

\\ \hline

&
\begin{matrix}
1      & z      &        \\
       & \ddots & \ddots \\
       &        & 1
\end{matrix}
\end{array}
\right),
\]
where the top left $(m_{k+1}-1)\times(m_{k+1}-1)$ block is lower triangular with diagonal entries equal to $z$, and the lower right $(k-m_{k+1}+1)\times (k-m_{k+1}+1)$ block is upper triangular with diagonal entries equal to $1$. Consequently,
\[
\det N[[k+1]-\set{m_{k+1}},[k]]=z^{m_{k+1}-1},
\]
which translates to
    \spliteq{}{
    &\Pr\left(\pi_{k+1}=m_{k+1}\,|\,\set{\pi_1,\ldots,\pi_{k}}=[k+1]-\set{m_{k+1}}\right)\propto z^{(k+2-m_{k+1})p},
    \\
    &\Pr\left(\pi_{k+1}=k+2\,|\,\set{\pi_1,\ldots,\pi_{k}}=[k+1]-\set{m_{k+1}}\right)\propto 1,}
    and all other probabilities are zero.
Note that this uniquely determines $m_{k+2}$: if $\pi_{k+1}=m_{k+1}$, then $m_{k+2}=k+2$, and if $\pi_{k+1}=k+2$, then $m_{k+2}=m_{k+1}$. By changing variables to $\delta_k=k-m_k$, this becomes a Markov chain
\[
\delta_{k+1}=\begin{cases}
    \delta_k+1 & \text{with probability }\dfrac1{1+z^{(1+\delta_k)p}}\\
    0 & \text{otherwise}
\end{cases}.\]
From this, one can bound the height of the excursions of $\delta_k$ from $0$:
\[
\Pr(\delta_{k+t}=t \,|\,\delta_k=0)
=\prod_{j=1}^t\pare{1 + z^{jp} }^{-1}.
\]
Note that $e^{-1}\le z^{jp}\le1$ is satisfied for $j\le1/{(p\log(1/z))}=n^\alpha$, so for $t> n^\alpha$ we have
\[
\Pr(\delta_{k+t}=t \,|\,\delta_k=0)
\le\prod_{j=1}^{\ceil{n^\alpha}}\pare{1 + z^{jp} }^{-1}
\le\pare{\frac e{e+1}}^{n^\alpha}
\le\exp\pare{-\frac{n^\alpha}{e+1}}.
\]
Stated differently, each excursion $\delta_k$ makes from a given $0$ value has length at most $n^\alpha$ with probability at least $1-e^{-n^\alpha/(e+1)}$. By applying a union bound over starting points of the excursion, we have that the maximum of the $\delta_k$ (and therefore $\delta_n$) is bounded by $n^\alpha$ with probability $ne^{-n^\alpha/(e+1)}$. Finally, $\delta_n\le n^\alpha$ is equivalent to $m_n=\pi_n\ge n-n^\alpha$.
\end{proof}

\begin{theorem}\label{a11}
Let $Q$ be as in \Cref{a18} for $\alpha\in(0,1)$ and $P\sim\RandPP_p(Q)$. Then, for $PQ=LU$, we have
\[\Pr\pare{\maxnorm U\ge\exp\pare{\frac{n^{1-\alpha}-1-n^{-\alpha}}p}}\ge1-ne^{-n^\alpha/(e+1)}.\]
In particular, with $\alpha=(\log 8 +\log\log n)/\log n$, we have
\[\Pr\pare{\maxnorm U\ge \exp \pare{\frac{n-16 \log n}{8 p \log n} } }\ge 1 - n^{1-8/(e+1)} .\]
\end{theorem}
\begin{proof}
Because $PQ$ is orthogonal, we have for
$PQ=LU$ that 
\[U_{nn}^{-1}=(U^{-1})_{nn} = ((PQ)^{-1})_{nn}=((PQ)^\top)_{nn}=(PQ)_{nn},\] 
which is the $\pi_n\toth$ entry of the last column of $Q$. By construction $Q \bm e_n$ is the unique unit vector (up to sign) orthogonal to the leading $n-1$ columns of $N$, and so must be proportional to
\[Q \bm e_n\propto\pmat{1&-z&(-z)^2&\cdots&(-z)^{n-1}}^\top.\]
Note that this vector has norm at least one. When $P\sim\RandPP_p(Q)$, \Cref{a19} says that $\pi_n$ is at least $n-n^\alpha$ with probability $1-ne^{-n^\alpha/(e+1)}$, and so the $\pi_n\toth$ entry of $Q \bm e_n$ is at most $z^{n-n^\alpha-1}$ in absolute value. Therefore, 
\[
\maxnorm U
\ge \abs{U_{nn}}
\ge z^{-(n-n^\alpha-1)}
 =\exp\pare{\frac{n^{1-\alpha}-1-n^{-\alpha}}p}\]
with probability $1-ne^{-n^\alpha/(e+1)}$.
\end{proof}

Because $PQ$ is orthogonal, $\maxnorm{PQ}\le1$, and so $\growth(PQ)\ge \|U\|_{\max}$.

\section{Quasi-Polynomial Growth of Rook Pivoting}
Here, we construct a sequence of matrices for which Gaussian elimination with rook pivoting exhibits growth $n^{(\frac14-o(1))\log_2 n}$. This is nearly within a factor of two of the exponent in Foster's upper bound $n^{\frac{3\log2}4\log_2 n}\approx n^{0.52\log_2 n }$ \cite{b19}.
Our construction is recursive:
for each matrix $A$ exhibiting large growth under rook-pivoting, we construct a larger rook-pivoted matrix $B$ with smaller entries which reveals $A$ in the course of performing eliminations. The key strategy for constructing such a gadget is focusing on the block-nonzero pattern
\[B = \bmat{W&X\\Y&0}.\]
This way, $A$ corresponds to the factorization $A=YW^{-1}X$. We take $W$ to be a multiple of the identity matrix, so that the behavior under successive elimination steps is easy to determine explicitly. We take $X$ to be a random matrix, so that its entries (and, when $W = \alpha I$, the entries of $Y = \alpha A X^{-1}$ as well) are small via a concentration argument, thereby ensuring $B$ is rook pivoted.

Our construction in the next section uses a similar idea, and so we present the concentration results needed for this section and the next here. Both concern properties of Haar-distributed matrices, though this section only makes use of the first bound \Cref{a20} for $Y=I$.
\begin{lemma}\label{a21}
Fix any $X\in\R^{n\times p}$, $Y\in\R^{n\times q}$, and $M\in\R^{n\times n}$. Let $\delta \in (0,1)$ and $L=\log(2pq/\delta)$.
If $Q\sim\Haar(\orthog(n))$, then
\eq{\label{a20}
\Pr\pare{
\maxnorm{X^\top QY}
\le
\norm X_{1\to2}\norm Y_{1\to2}
\sqrt{\frac{2L}{n}}}\ge1-\delta
}
and
\eq{\label{a22}
\Pr\pare{\maxnorm{X^\top Q^\top MQY-\frac{\tr M}{n}X^\top Y}\le C\norm X_{1\to2}\norm Y_{1\to2}\pare{\frac{\|M\|_F \sqrt L}n+\frac{\norm ML}n}}\ge1-\delta}
for an absolute constant $C$.
\end{lemma}
\begin{proof}
Consider an arbitrary pair of columns \(\bm x=X \bm e_i\) and \(\bm y=Y \bm e_j\). By rotational invariance,
\({\bm x^\top Q \bm y}/(\norm {\bm x} \norm {\bm y})\)
is distributed as the first coordinate of a vector uniformly distributed on the unit sphere \(\sph^{n-1}\), which, by a standard concentration result, satisfies
\[\Pr\pare{\frac{\abs{\bm x^\top Q \bm y}}{\norm{\bm x}\norm{\bm y}}>\frac{\gamma}{\sqrt{n}}}\le2e^{-\gamma^2/2}\]
for every \(\gamma>0\), see, e.g., \cite[Lemma 2.2]{b46}.
Taking \(\gamma =\sqrt{2L}\) and applying a union bound over the \(pq\) pairs of columns of $X$ and $Y$ gives the first result \Cref{a20}.

For the second estimate, express $\bm x$ and $\bm y$ in an orthonormal basis for their span, e.g., $\bm x=E \bm a$ and $\bm y=E \bm b$ for $E^\top E=I$ and $\col(E)=\spn\set{\bm x, \bm y}$. Then $Z=QE$ is a Haar-distributed sample from the Stiefel manifold of order at most two, and
\[\bm x^\top Q^\top MQ \bm y
=\bm a^\top Z^\top MZ \bm b
=\vvec(Z)^\top\frac{\bm a \bm b^\top\otimes M+\bm b \bm a^\top\otimes M^\top}2\vvec(Z).\]
Set $\wt M=\frac{\bm a \bm b^\top\otimes M+\bm b \bm a^\top\otimes M^\top}2$. In particular, we may apply the Stiefel Hanson-Wright inequality
\cite[Theorem 7.1(1)]{b47}, which gives
\[
\Pr\pare{ \abs{\vvec(Z)^\top\wt M\vvec(Z)-\frac1{n}\tr(\wt M)}\ge t }\le 2\exp\pare{-c\min\pare{ \frac{(n-2)^2t^2}{\smagn{\wt M}_F^2} , \frac{(n-2)t}{\smagn{\wt M}} }}\]
for $c=7/10000$. Observe that $\tr(\wt M)= \bm x^\top \bm y\tr(M)$, $\smagn{\wt M}_F\le\norm \bm x\norm \bm y\norm M_F$, and $\smagn{\wt M}\le\norm \bm x\norm \bm y\norm M$. Setting $t$ to be a sufficiently large multiple of $\norm \bm x\norm \bm y\pare{\frac{\magn M_F\sqrt L}n+\frac{\magn ML}n}$ and applying a union bound over \(pq\) pairs of columns of $X$ and $Y$ gives \Cref{a22}.
\end{proof}

\begin{theorem}\label{a12}
There is a universal constant $C$ such that for each $n\in\N$, there is a rook pivoted matrix $A\in\RP_n(\R)$ with $\growth(A)\ge n^{\frac14\log_2(n)-C\log\log n}$.
\end{theorem}
\begin{proof}
It suffices to prove the theorem statement when $n$ is a power of two by considering the matrix $\maxnorm{A}I_{n-2^{\floor{\log_2 n}}}\oplus A$ where $A\in\RP_{2^{\floor{\log_2(n)}}}(\R)$.
Let $Q_0=\bmat1$, $Q_1=I_2$, and $Q_k\sim\Haar(\orthog(2^k))$ be independent random Haar orthogonal matrices for $k\ge2$. Define $s_k\in\R_+$ recursively by $s_0=1$ and
\[s_k=4\sqrt{\frac k{2^k}}\max(s_j:j<k).\] 
Define $X_k\in\Mat_{2^k}(\R)$ recursively by $X_0=\bmat1$ and
\[X_{k+1}=\bmat{s_kI&Q_k^\top \\ -X_kQ_k&0}.\]
Our construction is simply the matrix $X_{K+1}$ for $K=\log_2(n)-1$, conditioned on a particular positive-probability event. Our analysis is in two parts: showing that $X_{K+1}$ has the desired growth and that $X_{K+1}$ is rook-pivoted with positive probability. Both parts require control on the row norms of $X_k$, i.e., the matrix norm $\magn{X_k}_{2\to\infty}$. To this end, note by definition that
\[\magn{X_{k+1}}_{2\to\infty}^2
\le\max\pare{\magn{X_kQ_k}_{2\to\infty}^2,s_k^2+1}
\le\max\pare{\magn{X_k}_{2\to\infty}^2,2\max(s_j:j\le k)^2}\]
so by induction,
\[\magn{X_k}_{2\to\infty}\le\sqrt2\max(s_j:j<k).\]
Furthermore, $s_k$ is eventually decreasing so $\sup_{k\in\N}s_k=:C$ is finite. We therefore have \[\maxnorm{X_{K+1}}\le\norm{X_{K+1}}_{2\to\infty}\le\sqrt2C.\]
On the other hand, observe that applying $2^k$ steps of Gaussian elimination to $X_{k+1}$ results in $\frac1{s_k}X_k$, and thus Gaussian elimination on $X_{K+1}$ eventually results in the scalar $(s_Ks_{K-1}\cdots s_1s_0)^{-1}$. Additionally, $s_k\le4C\sqrt{\frac k{2^k}}$ so
\[
\growth(X_{K+1})
\ge\frac{(s_Ks_{K-1}\cdots s_1s_0)^{-1}}{\maxnorm{X_{K+1}}}
\ge\frac{ 2^{\frac14(K^2+K)} }{ (4C)^{K+2}\sqrt{K!} }
\ge2^{\frac14K^2-(K+2)\log_2(4CK)},\]
which establishes our growth claim. We now turn our attention to pivoting.
Let $\mathcal E_k$ be the event that ($\maxnorm{X_kQ_k}\le s_k$ and $\maxnorm{Q_k}\le s_k$). Observe that the intersection $\mathcal E_0\cap\cdots\cap\mathcal E_{K}$ implies that $X_{K+1}$ is rook-pivoted. $\mathcal E_0$ and $\mathcal E_1$ occur by selection of $Q_0=\bmat1$, $Q_1=I_2$ and $s_0=1$, $s_1=4/\sqrt2$.
For each $k\ge2$, conditioned on the values of $Q_1,\ldots,Q_{k-1}$, the matrix $X_k$ is a fixed matrix with $\norm{X_k}_{2\to\infty}=\norm{X_k^\top}_{1\to2}\le\sqrt2\max(s_j:j<k)$. \Cref{a21} therefore gives
\spliteq{}{
\Pr\pare{ \maxnorm{X_kQ_k}\ge s_k }&
=\Pr\pare{ \maxnorm{X_kQ_k}\ge \sqrt{\frac{8k}{2^k}}\cdot\sqrt2\max(s_j:j<k) }
<\frac12,\\
\Pr\pare{ \maxnorm{Q_k}\ge s_k }&
\le\Pr\pare{ \maxnorm{Q_k}\ge 4\sqrt{\frac{k}{2^k}} }
<\frac12,
}
where the second inequality uses the trivial $\max(s_j:j<k)\ge1$. By a union bound, the probability of either event occurring is strictly less than 1.
Finally, $\Pr(\mathcal E_K\cap\cdots\cap\mathcal E_0)=\prod_{k=1}^K \Pr(\mathcal E_k\,|\,\mathcal E_{k-1}\cap\cdots\cap\mathcal E_0)>0$, which establishes $X_{K+1}\in\RP_{2^{K+1}}(\R)=\RP_n(\R)$ with positive probability.
\end{proof}

\section{Quasi-Polynomial Growth of Complete Pivoting}
Here, we produce a sequence of completely pivoted matrices with quasi-polynomial growth factor. Combined with the quasi-polynomial upper bound of approximately $n^{0.207 \log n}$ \cite{b5}, our result settles the asymptotic behavior of the growth factor under complete pivoting. Our lower bound has exponent $\frac1{16} \log_2 n$, which is nearly within a factor of two of the upper bound $0.207 \log n \approx 0.141 \log_2 n$. Our construction is similar in spirit to that of \Cref{a12}, with a few keys differences. For complete pivoting, we must bound the magnitude of all entries at each step, which requires a top-left block with sufficiently large pivot sizes (\Cref{a23}). In addition, to build a large growth factor, we require control on the norm of the input matrix, and, therefore, control on the norm of the output matrix as well. This is achieved by a block $4\times 4$ structure (\Cref{a24}) whose norm is well-controlled (\Cref{a25}).

First, we argue the existence of an orthogonal matrix, determined in \Cref{a23}, with pivots of sufficiently large magnitude under complete pivoting via a probabilistic argument. \Cref{a26,a27,a28} are preparation for this. Then, given this matrix, we define our recursive structure, \Cref{a24}.

\begin{lemma}[{\cite[Theorem 1.3]{b48}}]\label{a26}
If $G$ is an $n \times n$ Gaussian matrix, then
\[\Pr\pare{\norm{G^{-1}}_F \le \min\set{n/t,\sqrt{n/t}}} \le 2 e^{-c t^2}.\]
for all $t>0$ and some absolute constant $c$.
\end{lemma}
\begin{lemma}[{\cite[Lemma 4.1]{b49}}]\label{a27} Let $n \ge m \ge 2$. If $G$ is an $n\times m$ Gaussian matrix, then
\[\Pr\pare{\sigma_m(G)\le t\sqrt n}\le\frac{(tn)^{n-m+1}}{\Gamma(n-m+2)}\]
for all $t>0$.
\end{lemma}
\begin{lemma}\label{a29}
Let $W\sim\Haar(\orthog(n))$. Then with probability at least $1-e^{-cn}$, every $m\times m$ submatrix $X$ of $W$ with $\sqrt{n} \le m\le\frac n2$ simultaneously satisfies
\[\magn{X^{-1}}_F^2\ge cm \sqrt{n}\]
for an absolute constant $c$.
\end{lemma}
\begin{proof}
We first prove the claim for a fixed \(m\times m\) submatrix \(X\),
and then take a union bound over all such submatrices. Let
\(G_1\in\R^{m\times m}\) and \(G_2\in\R^{(n-m)\times m}\) be independent
standard Gaussian matrices.  Then
\[
X\overset{d}{=}
G_1\bigl(G_1^\top G_1+G_2^\top G_2\bigr)^{-1/2},
\]
and therefore, for the appropriate coupling of $X$ and $(G_1,G_2)$,
\eq{\label{a30}
\norm{X^{-1}}_F
\ge
\sigma_m\left(\bmat{G_1\\G_2}\right)\norm{G_1^{-1}}_F.
}
Applying \Cref{a26} in dimension \(m \ge \sqrt{n}\), with
\(t=C \sqrt{n}\) for a large absolute constant $C$, gives
\eq{\label{a31}
\Pr\left(
\norm{G_1^{-1}}_F^2\le \frac{\sqrt{m}}{C n^{1/4}}
\right)
\le
2\exp\pare{-c_0 C^2 n},}
where $c_0$ is the constant from \Cref{a26}. On the other hand, \Cref{a27} applied to
\(\bmat{G_1\\G_2}\) gives
\eq{\label{a32}\Pr\pare{\sigma_m\pare{\bmat{G_1\\G_2}}\le t\sqrt n}\le\frac{(tn)^{n-m+1}}{\Gamma(n-m+2)}\le\pare{\frac{ent}{n-m+1}}^{n-m+1}\le(2et)^{n/2},}
where the last inequality used $m\le\frac n2$. Set $t=\frac1{100}$.
It follows from
\Cref{a30,a31,a32}
that
\[
\Pr\pare{
\norm{X^{-1}}_F^2<\frac{\sqrt{n} m}{100^2C^2}}
\le
2\exp\pare{-c_0C^2 n}+\pare{\frac{2e}{100}}^{n/2}.
\]
Now, choose $C$ such that $c_0 C^2> \log 4 + 1$. The total number of submatrices of an $n\times n$ matrix is at most $4^n$, so a union bound provides our desired result.
\end{proof}

\begin{lemma}\label{a28}
Let $W\in\orthog(n)$ be any orthogonal matrix.  For every $m\times m$ nonsingular submatrix $X$ of $W$, there exists a square nonsingular submatrix $Y$ of $W$ of order $\min(m,n-m)$ with\eq{\label{a33}\magn{(W/X)}_F\ge\magn{ Y^{-1} }_F.}
Moreover,\eq{\label{a34}\maxnorm{(W/X)}\ge\frac1{\sqrt{n-m}}.}
\end{lemma}
\begin{proof}
Decompose \[W=\bmat{X&W_{12}\\W_{21}&W_{22}}.\]
By the Schur formula for the inverse,
\(((W^{-1})_{22})^{-1}=(W/X).\) But since $W$ is orthogonal, $W^{-1}=W^\top$, so \[(W/X)=(W_{22}^\top)^{-1}.\]
Since $W_{22}$ is the submatrix of an orthogonal matrix, all of its singular values are at most 1. Thus, all singular values of $(W/X)$ are at least 1, so $(n-m)^2\maxnorm{(W/X)}^2\ge\magn{(W/X)}_F^2\ge n-m$, establishing \Cref{a34}.

Now consider $m\ge n/2$. Since taking the transpose does not change the Frobenius norm, we may take $Y=W_{22}$ to obtain \Cref{a33} with equality. Now consider $m\le n/2$. The $n-m$ singular values of $W_{22}$ are exactly the $m$ singular values of $X$ with $n-2m$ copies of 1. In particular, $\smagn{W_{22}^{-1}}_F\ge\smagn{X^{-1}}_F$, so we may take $Y=X$.
\end{proof}

\begin{lemma}\label{a23}
    For all sufficiently large $n$, there is $W\in\orthog(n)\cap\CP_n(\R)$ such that
    \eq{\label{a35}\maxnorm W\le C\sqrt{\frac{\log n}n}.}
If $W\at j$ is the residual matrix after $j$ steps of Gaussian elimination, then
\eq{\label{a36}\mmaxnorm{W\at j}\ge c\frac j{n^{5/4}}\qquad \text{for} \quad j = 0,\ldots,n-1.}
\end{lemma}
\begin{proof}
Let $\wh W \sim\Haar(\orthog(n))$. By \Cref{a21}, \Cref{a35} holds for $\wh W$ with high probability. By \Cref{a29}, with probability at least $1-e^{-cn}$, every square submatrix $Y$ with order in the interval $[\sqrt{n},n/2]$ satisfies
\eq{\label{a37} \|Y^{-1}\|_{F} \ge c n^{1/4} \sqrt{\dim Y}.}
Now, consider a realization $\wh W$ that satisfies both these properties, and let $P$ and $Q$ be permutation matrices such that $W = P \wh W Q \in \CP_n(\R)$. Note that $W \in \orthog(n)$ and the above bounds still hold for $\|W\|_{\max}$ and square submatrices of $W$ with order in $[\sqrt{n},n/2]$.

We break our analysis of $W^{(j)}$ into three regimes, depending on the index $j$. If $\sqrt{n} \le j \le n - \sqrt{n}$, then, by \Cref{a37} and \Cref{a33}, there exists a square submatrix $Y$ of order $\min\{j,n-j\}$ with
\[\mmaxnorm{W \at j} \ge \frac{\mmagn{W \at j}_{F}}{n-j} \ge  \frac{\magn{Y}_F}{n-j} \ge \frac{c n^{1/4}\sqrt{\min(j,n-j)}}{n-j} \ge \frac{cj}{n^{5/4}}. \]
If $j < \sqrt{n}$, then, by \Cref{a34}, 
\[\mmaxnorm{W \at j} \ge \frac1{\sqrt{n-j}} \ge \frac{1}{\sqrt{n}} \ge \frac{j}{n^{5/4}}.\]
Finally, if $j > n - \sqrt{n}$, then, again, by \Cref{a34},
\[\mmaxnorm{W \at j} \ge \frac{1}{\sqrt{n-j}} > \frac{1}{n^{1/4}} \ge \frac{j}{n^{5/4}}.\]
\end{proof}

\begin{definition}\label{a24}
For $A\in \GL_n(\R)$, let 
\[\widehat A=\frac{A}{\magn A}, \qquad \alpha = \frac{n^{1/4}}{\log n},  \qquad \beta =\frac{n^{1/4}}{\log^{3/2} n}, \qquad \gamma = \frac{n^{1/4}}{\log^2 n},\]
and note that $\alpha \gamma = \beta^2$. Let $B(A)$ denote the $4n\times 4n$ random matrix
\[B(A)=\begin{+bmatrix}[colspec=cc|cc]
I&0& \alpha Q^\top & \gamma \widehat A\\
0&I&\beta WQ^\top&0 \\
\hline
\alpha WQ&0&0&0\\
0&\beta W^\top&0&0\end{+bmatrix},\]
where $W$ is the particular orthogonal matrix of \Cref{a23} and $Q\sim\Haar(\orthog(n))$.
\end{definition}

It is useful to first record the residuals of $B(A)$ after a certain number of steps of Gaussian elimination. Let $E_k \in \R^{n \times (n-k)}$ be the trailing $n-k$ columns of the identity matrix and $P_k=I-E_kE_k^\top$ be the projection matrix onto the first $k$ coordinates. By \Cref{a0},
\eq{\label{a38}
B(A)\at k=\begin{+bmatrix}[colspec=cc|cc]
I_{n-k}&0& \alpha E_k^\top Q^\top & \gamma E_k^\top \widehat A\\
0&I_n&\beta WQ^\top&0 \\
\hline
\alpha WQE_k&0&-\alpha^2 WQP_kQ^\top&-\beta^2 WQP_k\widehat A\\
0& \beta W^\top&0&0\end{+bmatrix}\qquad \text{for} \quad 0\le k< n.
}
The $(4,3)$ block becomes non-zero only after the first identity block is completely eliminated,
\eq{\label{a39}B(A)\at{n+k}=\begin{+bmatrix}[colspec=c|cc]
I_{n-k}& \beta E_k^\top WQ^\top&0 \\
\hline
0&- \alpha^2 W&-\beta^2 WQ\widehat A\\
\beta W^\top E_k&-\beta^2 W^\top P_kWQ^\top&0\end{+bmatrix}\qquad \text{for} \quad  0\le k <n.}
After the second identity block is eliminated,
\eq{\label{a40}B(A)\at{2n}=-\begin{+bmatrix}[colspec=cc]
\alpha^2 W& \beta^2 WQ\widehat A\\
\beta^2 Q^\top&0\end{+bmatrix}.}
In both cases, we benefit from the cancellations $QQ^\top =I$ and $W^\top W=I$ once the projection $P_k$ becomes identity. After the remaining $n$ steps of Gaussian elimination, we are left with a single block 
\eq{\label{a41}B(A) \at {3n}= \gamma^2 \widehat A.}
In order for $B(A) \in \CP_{4n}(\R)$, we must have $B(A)\at{2n} \in \CP_{2n}(\R)$, $\mmaxnorm{B(A)\at k} =1$ for $k =0,\ldots,n-1$, and $\mmaxnorm{B(A)\at{n+k}} =1$ for $k = 0,\ldots,n-1$. Conditional on some properties of $A$, we show that, with high probability, this is the case.

\begin{lemma}\label{a25} For $n$ sufficiently large,
\(\alpha\le\magn{B(A)}\le 3\alpha\) surely.\end{lemma}
\begin{proof}
For the lower bound, simply inspect the $(3,1)$ block. For the upper bound, one can replace each block of $B(A)$ with its spectral norm and take the Frobenius norm of the resulting $4\times 4$ scalar matrix,
\[\alpha=\magn{\alpha WQ}\le\magn{B(A)}\le\magn{ \pmat{1 & 0 & \alpha & \gamma \\ 0&1&\beta&0\\ \alpha&0&0&0\\0&\beta&0&0 } }_F = \pare{2+2\alpha^2 + 2 \beta^2 + \gamma^2}^{1/2} \le 3\alpha.\]
\end{proof}

In the following three lemmas, we implicitly invoke \Cref{a21}, \Cref{a20} and \Cref{a22}, order $n$ times, where $\delta$ is some sufficiently small polynomial in $n$, say $\delta=n^{-10}$, so that the stated overall probability claims hold by a simple union bound.

\begin{lemma}\label{a42}
If $\norm A\ge\frac \alpha 4$ and $\maxnorm A\le1$, then $\maxnorm{B(A)\at k}=1$ for $0\le k<n$ with probability at least $1-o(1)$.
\end{lemma}
\begin{proof}
We start with the non-zero blocks in the lower-left quadrant,
\spliteq{}{
\mmaxnorm{\alpha WQE_k}&\le \alpha \mmaxnorm{WQ}\lesssim \alpha \sqrt{\frac{\log n}{n}} = o(1),
\\
\mmaxnorm{\beta W^\top}&\le \beta\mmaxnorm{W}\lesssim \beta \sqrt{\frac{\log n}{n}} = o(1).}
Next, consider the non-zero blocks in the top-right quadrant,
\spliteq{}{
\mmaxnorm{ \alpha E_k^\top Q^\top }
&
\le \alpha \maxnorm{ Q } \lesssim \alpha \sqrt{\frac{\log n}{n}} = o(1),\\
\mmaxnorm{ \gamma E_k^\top \widehat A }
&
\le \frac{\gamma \maxnorm{ A }}{\norm A} \le \frac{4\gamma}{\alpha} = o(1).}
Finally, consider the non-zero blocks in the lower-right quadrant,
\spliteq{}{
\mmaxnorm{\alpha^2 WQP_kQ^\top}
&\lesssim \alpha^2 \mmaxnorm{W\E QP_kQ^\top}
+\alpha^2 \mmaxnorm{ WQP_kQ^\top - W\E QP_kQ^\top }
\\
&\lesssim \frac{\alpha^2 k}{n}\sqrt{\frac{\log n}n}
+\alpha^2 \pare{ \frac{\sqrt k\sqrt{\log n}}n + \frac{\log n}n } =o(1),
\\
\maxnorm{\beta^2 WQP_k\widehat A}&\le \beta^2 \sqrt{\frac{\log n}{n}} = o(1).
}
Note that all entries are smaller in magnitude than the pivot, which is 1.
\end{proof}

\begin{lemma}\label{a43}
If $\norm A\ge\frac \alpha 4$ and $\maxnorm A\le1$, then $\maxnorm{B(A)\at {n+k}}=1$ for $0\le k<n$ with probability at least $1-o(1)$.
\end{lemma}
\begin{proof}
We start with the non-zero blocks in the off-diagonal part,
\spliteq{}{
\mmaxnorm{\beta E_k^\top WQ^\top}&\le \beta \maxnorm{WQ^\top}\lesssim \beta \sqrt{\frac{\log n}{n}} = o(1),\\
\mmaxnorm{\beta W^\top E_k}&\le \beta \maxnorm{W}\lesssim  \beta \sqrt{\frac{\log n}{n}} = o(1).}
Next, consider the top two blocks in the lower-right quadrant,
\spliteq{}{
\maxnorm{\alpha^2 W}&\lesssim  \alpha^2 \sqrt{\frac{\log n}{n}} = o(1),
\\
\mmaxnorm{\beta^2 WQ\widehat A}&\le \beta^2 \mmaxnorm{WQ\widehat A}\lesssim  \beta^2 \sqrt{\frac{\log n}{n}} = o(1).}
Finally,
\[
\mmaxnorm{\beta^2 W^\top P_kWQ^\top} \lesssim  \beta^2 \sqrt{\frac{\log n}{n}} = o(1).\]
Note that all entries are smaller in magnitude than the pivot, which is 1.
\end{proof}

\begin{lemma}\label{a44}
If \(A\in\CP_n(\R)\), $\maxnorm A\le1$, and $\norm A\ge\frac \alpha 4$,
then \(B(A)\at{2n}\in\CP_{2n}(\R)\) with probability at least $1-o(1)$.
\end{lemma}
\begin{proof}
In order to express the matrix $B(A)\at{2n+k}$, we decompose $W$ and $Q$ into additional blocks. For a fixed \(0\le k<n\), consider the decomposition
\[
W=\bmat{W_{1:}\\W_{2:}}=
\bmat{
W_{11}&W_{12}\\
W_{21}&W_{22}
},
\qquad
Q=\bmat{Q_{1:}\\Q_{2:}},
\]
where \(W_{11}\in\mathbb R^{k\times k}\). Note that \(W_{2:}=E_k^\top W\) and \(Q_{2:}=E_k^\top Q\).
Next, we split the first block row and column of \(B(A)\at{2n}\) according to
this further partitioning:
\begin{equation}
B(A)\at{2n}
=-
\begin{+bmatrix}[colspec=cc|c]
\alpha^2 W_{11}
&
\alpha^2 W_{12}
&
\beta^2 W_{1:}Q\widehat A
\\
\alpha^2 W_{21}
&
\alpha^2 W_{22}
&
\beta^2 W_{2:}Q\widehat A
\\
\hline
\beta^2 Q_{1:}^\top
&
\beta^2 Q_{2:}^\top
&
0
\end{+bmatrix}.
\label{a45}
\end{equation}
Let
\[
S=(W/W_{11})
=
W_{22}-W_{21}W_{11}^{-1}W_{12}.
\]
Applying the Schur complement formula directly to
\Cref{a45}, we obtain
\begin{align}
B(A)\at{2n+k}
&=
\bmat{
-\alpha^2 W_{22}
&
-\beta^2 W_{2:}Q\widehat A
\\
-\beta^2 Q_{2:}^\top
&
0
}
-
\bmat{
-\alpha^2 W_{21}
\\
-\beta^2 Q_{1:}^\top
}
\bigl(-\alpha^2 W_{11}\bigr)^{-1}
\bmat{
-\alpha^2 W_{12}
&
-\beta^2 W_{1:}Q\widehat A
}
\nonumber\\
&=
\bmat{
-\alpha^2 S
&
-\beta^2 S Q_{2:} \widehat A
\\[1mm]
-\beta^2 (Q_{2:}^\top - Q_{1:}^\top W_{11}^{-1} W_{12})
&
\gamma^2 Q_{1:}^\top W_{11}^{-1}W_{1:}Q\widehat A
}.
\label{a46}
\end{align}
Next, we simplify the bottom blocks of \Cref{a46}.
By the block inverse formula applied to $W$, and the fact that $W^{-1}=W^\top$, we have
\eq{\label{a47} S^{-1} = W_{22}^\top \qquad \text{and} \qquad W_{11}^{-1} W_{12} = -W_{21}^\top S.}
Using \Cref{a47},
\begin{align*}
Q_{2:}^\top - Q_{1:}^\top W_{11}^{-1} W_{12} &= Q_{2:}^\top + Q_{1:}^\top W_{21}^\top S \\ &=  Q_{2:}^\top W_{22}^\top S + Q_{1:}^\top W_{21}^\top S \\ &= Q^\top W_{2:}^\top S.
\end{align*}
Let
\eq{\label{a48} M = I - W^\top E_{k} S E_k^\top,}
or, in block coordinates,
\[M = \bmat{I & -W_{21}^\top S \\ 0 & I - W_{22}^\top S} = \bmat{I & W_{11}^{-1} W_{12} \\ 0 & 0}.\]
Then $\tr M = k$ and $Q_{1:}^\top W_{11}^{-1} W_{1:} = Q^\top M$. We may rewrite \Cref{a46} as
\eq{\label{a49} B(A)\at{2n+k} = \bmat{
-\alpha^2 S
&
-\beta^2 S Q_{2:} \widehat A
\\[1mm]
-\beta^2 (W_{2:} Q)^\top S
&
\gamma^2 Q^\top M Q\widehat A
}. }
Note that this formula holds when $k =0$, as, in that case, $S = W$ and $M = I - W^\top W = 0$.

Because $W \in \CP_{n}(\R)$, $|S_{11}| = \|S\|_{\max}$. What remains is to show that $\alpha^2 \|S\|_{\max}$ is the largest magnitude entry of \Cref{a49}. We begin with the top-right block, which we rewrite as
\[S Q_{2:} \widehat A = (E_k S^\top)^\top Q \widehat A.\]
Every column of $E_k S^\top$ is a row of $S$, and so $\norm{E_k S^\top}_{1 \to 2} \le \sqrt{n-k} \maxnorm S$. In addition, $\big\|\widehat A\big\|_{1 \to 2} \le \big\|\widehat A\big\| = 1$. Therefore, by \Cref{a20},
\[\mmaxnorm{S Q_{2:} \widehat A} = \mmaxnorm{(E_k S^\top)^\top Q \widehat A} \lesssim \sqrt{n-k} \maxnorm S \sqrt{\frac{\log n}{n}} =o\pare{\frac{\alpha^2}{\beta^2}} \maxnorm{S}.\]
Next, consider the lower-left block, which we rewrite as
\[(W_{2:} Q)^\top S = Q^\top W^\top E_k S.\]
Because $W^\top E_k$ has orthonormal columns, 
\[\norm{W^\top E_k S}_{1\to 2} = \max_{j \in [n-k]} \norm{S \bm e_j} \le \sqrt{n-k} \maxnorm S,\]
and so, by \Cref{a20},
\[\big\|(W_{2:} Q)^\top S\big\|_{\max} = \big\|Q^\top W^\top E_k S \big\|_{\max} \lesssim \sqrt{\log n} \maxnorm{S} =o\pare{\frac{\alpha^2}{\beta^2}} \maxnorm{S}. \]
Finally, we consider the lower-right block. By \Cref{a48} and the orthogonality of $W$,
\[ \norm{M}_F \le \sqrt{n} + \|S\|_F \le \sqrt{n} + (n-k) \norm{S}_{\max} \qquad \text{and} \qquad \norm M \le 1 + \norm S \le 1 + (n-k) \norm{S}_{\max}. \]
By \Cref{a22} with $X = I$, $Y = \widehat A$, and $M$ as itself,
\begin{align*}
    \maxnorm{Q^\top M Q \widehat A -\frac{k}{n} \widehat A } &\le C \pare{\frac{\pare{\sqrt{n}+(n-k)\maxnorm S}\sqrt{\log n}}n + \frac{\pare{1+(n-k)\maxnorm S}\log n}n} \\
    &\lesssim \sqrt{\frac{\log n}n} + \frac{(n-k)\maxnorm S \log n}n \\
    &\lesssim \maxnorm S \log n,
\end{align*}
as, by \Cref{a34}, $\maxnorm S \ge (n-k)^{-1/2}$. Because $\|A\| \ge \frac \alpha 4$ and, by \Cref{a36}, $\maxnorm{S} \ge c\frac k{n^{5/4}}$, we have
\begin{align*}
\big\|Q^\top M Q \widehat A \big\|_{\max} &\le \maxnorm{Q^\top M Q \widehat A -\frac{k}{n} \widehat A } + \maxnorm{\frac{k}{n} \widehat A} \\
&\lesssim \|S\|_{\max} \log n + \frac{4 k}{n \alpha} \\
&=o\pare{\frac{\alpha^2}{\gamma^2}}\maxnorm{S}.
\end{align*}
This concludes the analysis of $B(A)\at{2n+k}$ for $k = 0,\ldots,n-1$. Since $B(A)\at{3n} = \gamma^2 \widehat A$ and $\widehat A \in \CP_{n}(\R)$, we have $B(A) \at{2n} \in \CP_{2n}(\R)$.
\end{proof}

The above three lemmas hold simultaneously with positive probability, yielding:

\begin{proposition}\label{a50}
For sufficiently large $n$, if there exists $A \in \CP_n(\R)$, $\maxnorm{A} =1$, with $\norm A \ge \frac \alpha 4$, then there is a realization of $Q \in\orthog(n)$ for which $B(A) \in \CP_{4n}(\R)$, $\maxnorm{B(A)} =1$, $\alpha \le \norm{B(A)} \le 3 \alpha$, and
\eq{\label{a51} \growth(B(A)) \ge \frac{\gamma^2}{\norm A} \growth(A).}
\end{proposition}

\begin{theorem}\label{a13}
There is a universal constant $C$ such that for each $n\in\N$, there is a completely pivoted matrix $A\in\CP_n(\R)$ with $\growth(A)\ge n^{\frac1{16}\log_2 n-C\log\log n}$.
\end{theorem}
\begin{proof}
It suffices to prove our desired result for geometrically spaced $n$, for then it would then hold for all $n$ simply by padding the resulting matrices with the identity matrix, $I \oplus A$. 

Let $n_0$ be sufficiently large so that \Cref{a50} holds. Let 
\[\alpha(n) = \frac{n^{1/4}}{\log n}, \qquad \gamma(n) = \frac{n^{1/4}}{\log^2 n}, \qquad
s = \frac{\alpha(n_0)-1}{n_0 - \alpha(n_0)}, \qquad \text{and} \qquad A_0 = \frac{I + s \one \one^\top}{1+s}. \]
We have $\maxnorm{A_0} =1$, $\norm{A_0}= \alpha(n_0)$, and
\[A_0\at j = \frac{1}{1+s} \pare{I + \frac{s}{1+j s} \one \one^\top} \qquad \text{for} \quad j = 0,\ldots,n-1.\]
Therefore, $A_0 \in \CP_{n_0}(\R)$ and $\growth(A_0)=1$. Now, let $n_k = 4^k n_0$ and note that, if there exists an $A_{k}\in \CP_{n_k}(\R)$, $\|A_{k}\|_{\max} =1$, with $\frac14 \alpha(n_k) \le \|A_k\| \le 3 \alpha(n_k)$, then we may choose a realization $A_{k+1} = B(A_k)$ provided by \Cref{a50} such that $A_{k+1} \in \CP_{n_{k+1}}(\R)$, $\maxnorm{A_{k+1}}=1$, and
\[\frac1{4} \alpha(n_{k+1}) \le \alpha(n_k) \le \norm{A_{k+1}} \le 3 \alpha(n_k) \le 3 \alpha(n_{k+1}).\]
The growth factor of $A_{k+1}$ is at least
\[\growth(A_{k+1}) \ge \frac{\gamma(n_k)^2}{3 \alpha(n_k)} \growth(A_k) .\]
Applying this repeatedly, and noting that $k = \frac12\pare{\log_2 n_k - \log_2 n_0}$, we obtain 
\[\growth(A_{k}) \ge \prod_{j =0}^{k-1}\frac{\gamma(n_j)^2}{3 \alpha(n_j)} = \prod_{j=0}^{k-1} \frac{\pare{4^jn_0}^{\frac14} }{3\log^3 (4^j n_0)} \ge \frac{4^{\frac{k(k-1)}{8}} n_0^{\frac k4}}{3^{k}\log^{3k}(n_k)} = \frac{n_k^{\frac k4}}{3^k 4^{\frac{k(k+1)}8}\log^{3k}(n_k)} = n_k^{\frac1{16} \log_2 (n_k) - C \log \log n_k}\]
for some constant $C$ and all $k$ sufficiently large.
\end{proof}

\begin{remark}
    One can construct $A$ achieving the bound of \Cref{a13} so that it is the unique element in $\set{P AQ}\cap\CP_n(\R)$, i.e., there is no freedom in implementing complete pivoting. To do this, just multiply on both sides by $\diag(1,e^{-\eps},e^{-2\eps},\ldots)$ for sufficiently small $\eps>0$.
\end{remark}

\section*{Acknowledgements}
The authors would like to thank Louisa Thomas for improving the style of presentation. The second author would like to thank Alan Edelman for years of conversation and collaboration on the growth factor problem in Gaussian elimination, as well as illuminating conversations with Alex Townsend, Nick Higham, and Nick Trefethen on the subject. This material is based upon work supported by the Institute for Advanced Study and the National Science Foundation under Grant No. DMS-1926686 -- the second author is grateful to the IAS for their membership for the 2021--2022 academic year, where a number of preliminary results and ideas which inspired this final manuscript were developed. This material is based upon work supported by the National Science Foundation under Grant No. DMS-2513687. Some of the research presented here is the result of human--AI interaction, see the AI usage statement below.

\section*{AI Usage Statement}
This manuscript blends purely human-based mathematics with some mathematics that involved the use of ChatGPT Sol 5.6 and Claude Fable 5. This entire manuscript, including the results and proofs, was written by the authors. We aim to give a clear and detailed picture of AI involvement, section by section:
\begin{itemize}
\item Section 1 (human): No AI involvement.
\item Section 2 (human/AI): Authors proved  results regarding sparse matrices and partial pivoting in 2023. Interaction with AI improved the human-only results. 
\item Section 3 (human/AI): AI provided explanations of DPPs, examples of stochastic domination, and answered questions regarding the strong Rayleigh property. Authors had the idea of using a projection DPP, a Sylvester Hadamard matrix, and Hadamard matrices. Interaction with AI led to the proof that the choice of row permutation does not affect the Frobenius norms of $L$ and $U$.
\item Section 4 (human/AI): Authors had a proof of NP-hardness for a slightly different definition of growth factor in 2023. Interaction with AI led to a hardness proof for $\growth(\cdot)$, which does not appear in the paper. Inspired by this proof, which shares many similarities with the original human-only proof, authors proved, without AI, the construction that appears in the manuscript.
\item Section 5 (human/AI): AI was given the measure and prompted to compute expectations using determinant formulas. These were unsatisfactory. AI suggested a Jordan block after a prompt for a proof of boundedness. Authors performed simulations, which suggested that the expectation was tail dominated (AI confirmed this). Authors instead considered high probability bounds, still with a Jordan block. Authors mistakenly thought taking fixed $z$ near $1$ worked. Interaction with AI led to the suggestion of $z = 1 - 1/\sqrt{n}$. The above was for $p =2$; authors generalized this without AI.
\item Section 6 (human): No AI involvement. The earliest proof that the growth factor under rook pivoting is quasi-polynomial was completed in 2022.
\item Section 7 (human/AI): Interaction with AI led to a quasi-polynomial lower bound construction. This interaction strongly followed the authors' research program/architecture for proving that complete pivoting has quasi-polynomial growth, which was a human-only endeavor from 2022 until a week prior to the first draft of this manuscript. Indeed, the resulting lower bound follows the human-generated road map that was provided. The resulting construction was quite complicated and hard to follow. Authors produced, without AI, the much simpler and intuitive construction that appears in the manuscript. This research program for the growth factor under complete pivoting could not have been completed in such a short time period without the involvement of AI.
\end{itemize}
 The authors welcome questions regarding the role of AI in this work and comments regarding the best methodology for disclosing AI interactions, as no universal norms seem to currently exist yet.

\bibliographystyle{alpha}
\bibliography{outbib}

\end{document}